\documentclass[12pt]{amsart}
\usepackage{latexsym,amsmath, amscd, amssymb, amsthm, euscript, mathrsfs}
\usepackage[all]{xy}
\usepackage{hyperref}
\usepackage[cal=boondoxo, scr=boondoxo]{mathalfa}

\usepackage[capitalize]{cleveref}

\newtheorem{thm}[subsection]{Theorem}

\newtheorem{thm/def}[subsection]{Theorem/Definition}

\newtheorem{lem}[subsection]{Lemma}

\newtheorem{prop}[subsection]{Proposition}

\theoremstyle{definition}

\theoremstyle{definition}

\theoremstyle{definition}

\newtheorem{rem}[subsection]{Remark}

\newtheorem{example}[subsection]{Example}
\numberwithin{equation}{subsection}

\newtheorem{pg}[subsection]{}
\crefname{pg}{}{}
\creflabelformat{pg}{(#2#1#3)}

\newcommand{\scr}[1]{\mathbf{\EuScript{#1}}}

\newcommand{\Sp}{\text{\rm Spec}}

\newcommand{\mls}{\mathscr}

\newcommand{\perf}{\mathrm{perf}}
\newcommand{\lotimes}{\otimes ^{\mathbf{L}}}
\newcommand{\Perf}{\mls D_\perf}
\newcommand{\uPerf}{\underline {\mls D}_{\mathrm{perf}}}

\newcommand{\mdet}{\mathrm{det}}
\newcommand{\zar}{\mathrm{zar}}
\newcommand{\Def}{\mathrm{Def}}
\newcommand{\SPcite}[1]{\cite[\href{https://stacks.math.columbia.edu/tag/#1}{Tag #1}]{stacks-project}}
\newcommand{\dg}{\mathrm{dg}}
\newcommand{\cof}{\mathrm{cof}}

\newcommand{\DK}{\mathrm{DK}}
\newcommand{\op}{\mathrm{op}}
\newcommand{\strict}{\mathrm{strict}}

\newcommand{\simto}{\stackrel{\sim}{\to}}
\newcommand{\BaseSite}{\mathbf{S}}

\AtBeginDocument{%
   \def\MR#1{}
}

\begin{document}

\title{Deformation theory of perfect complexes and traces}
\author{Max Lieblich and Martin Olsson}

\begin{abstract}
    We show that the deformation theory of a perfect complex   and that of its determinant are related by the trace map,  in a general setting of sheaves on a site.  The key technical step, in passing from the setting of modules over a ring where one has global resolutions to the general setting, is achieved using $K$-theory and higher category theory.
\end{abstract}

\maketitle

\section{Introduction}

\begin{pg}\label{P:slogan}
Our goal in this article is to prove the following piece of folklore.

\begin{thm}[Folk theorem, informally stated]\label{thm:folk}
  If $E$ is a perfect complex on an algebraic stack $X$ with a first-order thickening $X\hookrightarrow X'$, then the trace of the obstruction class of $E$, with respect to the thickening, is the obstruction class for the determinant $\det E$ of $E$. Moreover, the trace map from the torsor of deformations of $E$ to the torsor of deformations of $\det E$ coincides with the determinant map.
\end{thm}

This is a geometric generalization of something familiar from a first multivariable analysis course: the derivative of the determinant map on the space of matrices is the trace function.

In our earlier work \cite{HLT, LO, LOstronger}, we needed this result in a level of generality not explicitly available in the literature. Indeed, as we discuss in \cref{P:history}, there are many incarnations of \cref{thm:folk}, in the context of both classical and derived algebraic geometry. However, no source of which we are aware treats the crucial case of a  scheme over a mixed characteristic base ring, or a gerbe over such a scheme. Some of the classical arguments written over the complex numbers generalize easily to our needs, while others do not. Moreover, the literature discussing this result in derived algebraic geometry  generally starts with a blanket characteristic $0$ assumption, and it is not apparent to us which results generalize as written. At the very least, the literature as written is inadequate for the applications that presently exist.

In this article we prove this folk theorem for deformations of perfect complexes  in a ringed topos, which suffices for all applications of which we know.
\end{pg}

\begin{pg}\label{P:1.1}
Let $\BaseSite$ be a site and let $\mls O'\rightarrow \mls O$ be a surjective morphism of rings on $\BaseSite$ with square-zero kernel $K$.

Let $E\in D^b(\mls O)$ be a perfect complex of $\mls O$-modules on $\BaseSite$  \cite[Tag 08G4]{stacks-project}.  A deformation of $E$ to $\mls O'$ is a pair $(E', \sigma )$, where $E'\in D^-(\mls O')$ is a complex and $\sigma :E'\lotimes _{\mls O'}\mls O\simto E$ is an isomorphism in $D(\mls O)$.  Such a complex $E'$ is automatically perfect (see for example \cite[3.2.4]{deformingkernels}).  It is well-known, and documented in various levels of generality in the literature, that then the following hold.
\begin{enumerate}
    \item [(i)] There is an obstruction $\omega (E)\in \mathrm{Ext}^2_{\mls O}(E, E\lotimes _{\mls O}K)$ which vanishes if and only if there exists a deformation of $E$ to $\mls O'$.
    \item [(ii)] If $\omega (E) = 0$ then the set of deformations of $E$ to $\mls O'$ form a torsor under $\mathrm{Ext}^1_{\mls O}(E, E\lotimes _{\mls O}K)$.  We denote this action by
    $$
    (E', \sigma )\mapsto \alpha *(E', \sigma )
    $$
    for a deformation $(E', \sigma )$ of $E$ and a class $\alpha \in \mathrm{Ext}^1_{\mls O}(E, E\lotimes _{\mls O}K)$.
    \item [(iii)] If furthermore $\mathrm{Ext}^{-1}_{\mls O}(E, E)=0$  then the set of automorphisms of any deformation of $E$ to ${\mls O}'$ is canonically isomorphic to $\mathrm{Ext}^0_{\mls O}(E, E\lotimes _{\mls O}K).$ 
\end{enumerate}

The purpose of this article is to elucidate the compatibility of these three facts with traces. In the course of the article we will also review the construction of the obstruction $\omega (E)$ in the required degree of generality, as well as points (ii) and (iii).  We will furthermore explain how to modify (iii) in the case when the vanishing of negative $\mathrm{Ext}$-groups does not hold. 
\end{pg}

\begin{rem}
One can generalize the definition of the obstruction to bounded above complexes, which are not necessarily perfect.  This can be done using the construction of Gabber, discussed in subsection \ref{S:section9}, or in the more general context of spectral algebraic geometry as discussed in \cite[\S 16.2]{LurieSAG}.
\end{rem}

\begin{pg} 
For a perfect complex $E$ as above we can consider its determinant $\det (E)$, which is an invertible ${\mls O}$-module.  This is again documented in various level of generality in the literature; for example in \cite{Knudsen}. We explain in this article how to define the determinant in our general setting.  On the other hand we can also consider the trace map
$$
\mathrm{tr}:\mathrm{Ext}^i_{\mls O}(E, E\lotimes _{\mls O}K)\rightarrow H^i(\BaseSite, K),
$$
defined in \cite[V, 3.7.3]{Illusie}.  The main result of this article is the compatibility of determinants and traces in the following sense:
\end{pg}

\begin{thm}\label{T:maintheorem}
{\rm (i)} $\mathrm{tr}(\omega (E)) = \omega (\det (E))$ in $H^2(\BaseSite, K).$

{\rm (ii)} If $(E', \sigma )$ is a deformation of $E$ and $\alpha \in \mathrm{Ext}^1_{\mls O}(E, E\lotimes _{\mls O}K)$ is a class, then 
$$
\det (\alpha *(E', \sigma )) = \mathrm{tr}(\alpha )*(\det (E'), \det (\sigma )).
$$

{\rm (iii)} If furthermore we have  $\mathrm{Ext}^{-1}(E, E) = 0$ then for a deformation $(E', \sigma )$ the map on automorphism groups 
$$
\mathrm{Ext}^0_{\mls O}(E, E\lotimes _{\mls O}K)\rightarrow \mathrm{Ext}^0_{\mls O}(\det (E), \det (E)\lotimes _{\mls O}K)\simeq H^0(\BaseSite, K)
$$
induced by the determinant agrees with the trace map.
\end{thm}

\begin{pg}\label{P:history}
 This compatibility seems to be well-known to experts and appears in the literature in the case of complexes of coherent sheaves admitting a global resolution in \cite{HuybrechtsThomas, Langholf, Thomas}. The  case of perfect complexes on quasi-projective schemes over a field of characteristic $0$ (which themselves always admit global resolutions) also appears in \cite[Proposition 3.2]{STV}.  As we explain in \cref{S:section10}, these cases of complexes admitting global resolutions also follow from the additivity of traces for morphisms in the filtered derived category \cite[V, 3.7.7]{Illusie}. The results in \cite[Chapter 7, \S 3.3]{GRII}, which concern the cotangent complex of the stack of perfect complexes for derived schemes over fields of characteristic $0$, are also closely related to the work in this article.  Our approach in this article is close in spirit to \cite{GRII, STV}.

\end{pg}
\begin{pg}
Fundamentally, \cref{T:maintheorem} is a reflection of a more basic statement in the context of ``formal moduli problems'' in the sense of Lurie \cite[Chapter IV]{LurieSAG}.  Both perfect complexes and line bundles form such moduli problems and the determinant map defines a morphism between them for which the induced map on tangent complexes is the trace map.  While we do not use the language of formal moduli problems, this framework captures the approach taken here.

There are two main issues to be dealt with in proving \cref{T:maintheorem}.  The first is the definition of the determinant of a perfect complex. In classical treatments, such as \cite{Knudsen}, one presents a complex locally using a resolution, takes the alternating tensor product of the determinants of the sheaves in the complex, and then has to argue that this globalizes and enjoys various good properties.  This approach to defining the determinant is difficult to work with in the general context of this article.   The second issue is that the deformation problem we are concerned with is fundamentally higher-categorical in nature.  We should consider not only complexes and isomorphisms between them, but homotopies and higher homotopies.
Both these issues are addressed by considering the problem from an $\infty $-categorical perspective.
\end{pg}

\begin{pg} 
In order to understand the obstruction class $\omega (E)$ we will employ the following basic idea, discussed at length in \cite[0.1.3.5 and surrounding text]{LurieSAG}.  For a map of ${\mls O}$-modules $\rho :K\rightarrow J$ one can pushout ${\mls O}'$ along $\rho $ to get a new surjection
$$
{\mls O}'_\rho \rightarrow {\mls O}
$$
with kernel $J$, and equipped with a morphism ${\mls O}'\rightarrow {\mls O}'_\rho $.  If we could find an inclusion $\rho :K\hookrightarrow J$ of $K$ into an injective $J$ such that $E$ lifts to a perfect complex $E'_\rho $ over ${\mls O}'_\rho $, then the obstruction class can be understood as follows.  The pushout of ${\mls O}'$ along the composition
$$
K\rightarrow J\rightarrow J/K
$$
is isomorphic to ${\mls O}[J/K]$ (the ring of dual numbers on $J/K$), and we get by pushing out $E'_\rho $ along ${\mls O}'_\rho \rightarrow {\mls O}[J/K]$ a class in 
$$
\mathrm{Ext}^1_{\mls O}(E, E\lotimes J/K).
$$
The image of this class under the boundary map
$$
\mathrm{Ext}^1_{\mls O}(E, E\lotimes J/K)\rightarrow  \mathrm{Ext}^2_{\mls O}(E, E\lotimes K)
$$
arising from the short exact sequence
$$
0\rightarrow K\rightarrow J\rightarrow J/K\rightarrow 0
$$
is then the obstruction $\omega (E)$.
Unfortunately it is not always possible to find such an inclusion $\rho $.  However, we can always choose an inclusion $K\hookrightarrow J$ into an injective ${\mls O}$-module and consider the induced inclusion
$$
K\hookrightarrow I:= (\xymatrix{J\ar[r]^-{\mathrm{id}}& J}),
$$
where the complex $I$ on the right is concentrated in degrees $-1$ and $0$.  

Applying the Dold-Kan correspondence to $I$ we obtain an inclusion of simplicial $\mls O$-modules $K\hookrightarrow I_\bullet $, and we can form the pushout of $\mls O'$ along $K\rightarrow I_\bullet $ in the category of simplicial rings.  
 This leads us to consider perfect complexes over simplicial rings and their determinants, which is the context for our discussion of the determinant.
\end{pg}

\begin{pg} The article is organized as follows. 

In \cref{S:section2new} we review the basic definitions pertaining to the $\infty $-category  of modules over a sheaf of simplicial rings.  We then explain how to understand the fiber, in the sense of $\infty $-categories, of the reduction functor obtained from a surjection of simplicial algebras with square-zero kernel, such as that which arises in our deformation problem for complexes.

\cref{S:Ktheory} contains a brief review of the various approaches to $K$-theory and comparisons between them that will play a role in this article. 

\cref{S:section3,S:section4} are devoted to a discussion of the determinant functor from perfect complexes to a suitable Picard category of line bundles.  While classically one defines the determinant using resolutions and gluing, the $\infty $-categorical approach to the determinant is easier to work with in our context (in fact, we do not know how to define the determinant in the necessary generality without it).  The key point is that the (connective) $K$-theory of the category of perfect complexes is realized as the universal map to a grouplike $\mathbf{E}_\infty $-monoid from the $\mathbf{E}_\infty $-monoid of projective modules.  We can then use variants of Quillen's plus construction to describe the $K$-theory of perfect complexes in more explicit ways that allows us to define the determinant directly.  We then globalize the discussion by taking global sections in the $\infty $-categorical sense (homotopy limits).

As pointed out to us by Bhargav Bhatt, the construction of the determinant used  in this article also enables us to prove a compatibility with ring structure on $K$-theory.   We explain this in \cref{S:section5}.  The reader may wish to omit this section as it is not used in the rest of the article.

\cref{S:trace} gives a description of the trace map of \cite[V, 3.7.3]{Illusie} from an $\infty $-categorical perspective which will play a role in comparing it with the determinant map.

In \cref{S:section6} we prove the fundamental compatibility of the determinant and trace maps.  The main result is \cref{P:4.6.1}.  In the following two Sections \ref{S:section7} and \ref{S:section8} we then apply the theory to the deformation theory of perfect complexes and prove the theorems discussed in this introduction.
In  \cref{S:section7} we also verify the equivalence 
 of the definition of the obstruction to deforming a perfect complex  to one due to Gabber.

Finally in \cref{S:section10} we give an alternate proof of \cref{T:maintheorem} (i) in the case when one has global resolutions.  The approach in this section does not use $\infty $-categories but rather the filtered derived category and the compatibility of the trace map with passing to the associated graded proven in \cite[V, 3.7.7]{Illusie}.

We have also included an appendix wherein we establish the basic relationship between sheaves of dg-modules over the normalization of a simplicial ring and the $\infty $-category of modules in the sense of \cref{S:section2new}.  This result seems well-known to experts but we include it here for lack of a suitable reference.
\end{pg}

\subsection{Acknowledgments} The authors are grateful to Luc Illusie who stressed the importance of providing a proof of \cref{T:maintheorem}, and to Benjamin Antieau, Bhargav Bhatt, Akhil Mathew, and Joseph Stahl for helpful conversations and comments on preliminary versions of this article.  We note, in particular, that Bhatt suggested the contents of \cref{S:section5}.  The authors also thank Ofer Gabber for sharing his construction of the obstruction in \cite{Gabber}.
The authors thank the referee who suggested a number of corrections and clarifications which greatly improved the article.

During the work on this article, Lieblich was partially supported by NSF grant 
DMS-1600813 and a Simons Foundation Fellowship, and Olsson was partially
supported by NSF grants DMS-1601940 and DMS-1902251. Part of this work was done in Spring 2019 when the authors visited MSRI in Berkeley, whose support is gratefully acknowledged.

\subsection{Conventions}
We will use the language of $\infty $-categories as developed by Lurie in \cite{LurieHT, LurieSAG, LurieHA}, and differential graded (dg) categories as in \cite{toen}.

We will often pass from a stable $\infty $-category $\mls C$ to its underlying $\infty $-groupoid, which we will denote by $\mls C^\simeq $.  This is the $\infty $-category obtained from $\mls C$ by considering only morphisms which induce isomorphisms in the homotopy category (see \cite[1.2.5.3]{LurieHT}).

Throughout this article all simplicial rings considered will be commutative, and we usually omit the adjective ``commutative''.

For a ring $A$ and simplicial $A$-module $I_\bullet $ we will often consider the
simplicial ring $A[I_\bullet ]$ of dual numbers given by $A\oplus I_\bullet $ and multiplication 
$$
(a, i)\cdot (b, j) = (ab, aj+bi).
$$
There is a surjection $\pi :A[I_\bullet ]\rightarrow A$ sending $I_\bullet $ to $0$, and a retraction $A\rightarrow A[I_\bullet ]$ given by $a\mapsto (a, 0)$.

We write $\mathrm{Sp}$ for the $\infty $-category of spectra  and $\mathrm{Sp} ^{\geq 0}$ for the $\infty $-category of connective spectra (see \cite[0.2.3.10 and 0.2.3.12]{LurieSAG}).

\section{The $\infty $-category of perfect complexes}\label{S:section2new}

\subsection{Animated rings}

For a commutative ring $R$ we follow \cite[25.1.1.1]{LurieSAG} and consider an $\infty $-category $\text{CAlg}_R^\Delta $.  In \cite{LurieSAG} this is referred to as the $\infty $-category of simplicial commutative rings, but we prefer to reserve this term for simplicial commutative rings in the classical sense, and use the terminology of \cite[5.1.6 (3)]{CS}
and refer to $\text{CAlg}_R^\Delta $ as the $\infty $-category of \emph{animated rings}.  As noted in loc. cit. the $\infty $-category $\text{CAlg}_R^\Delta $ can be viewed as the $\infty $-category obtained from simplicial commutative $R$-algebras by inverting weak equivalences.  The category of animated rings can also be viewed as the $\infty $-category obtained by starting with the category of simplicial commutative $R$-algebras, endowing this category with the simplicial model category structure described in \cite[5.5.9.1]{LurieHT}, and applying the nerve to the subcategory of cofibrant-fibrant objects \cite[25.1.1.3]{LurieSAG}.  Because of these descriptions of $\text{CAlg}_R^\Delta $ we will use simplicial notation in describing objects of $\text{CAlg}_R^\Delta $ (e.g. $A_\bullet \in \text{CAlg}_R^\Delta $).

\subsection{Modules over animated rings}

For an animated $R$-algebra $A_\bullet \in \text{CAlg}_R^\Delta $ we have the associated stable infinity category of $R$-modules \cite[25.2.1.1]{LurieSAG}, which we will denote by $\mls D(A_\bullet )$ (in loc. cit. this category is denoted $\text{Mod}_{A_\bullet }$).

By \cite[I, 3.1.3]{Illusie}, for a simplicial $R$-algebra $A_\bullet $ the normalization $N(A_\bullet )$ is a strictly commutative differential graded algebra, and normalization defines a functor from $A_\bullet $-modules to differential graded $N(A_\bullet )$-modules.  
As noted in \cite[2.5]{KST}, by an argument similar to \cite[Proof of 7.1.2.13]{LurieHA}, this defines an equivalence between $\mls D(A_\bullet )$ and the $\infty $-category obtained from the category of dg-modules over $N(A_\bullet )$ by inverting quasi-isomorphisms (this is the approach taken for example in \cite[3.1]{ToenDAG}; see also \cite[1.1]{Shipley}).

\subsection{Topology}

Let $\Lambda $ be a ring and let $\mathbf{S}$ be a site.

\begin{pg}
As in \cite[\S 1.3.5]{LurieSAG} we can consider sheaves of animated $\Lambda $-algebras, defined as sheaves on $\mathbf{S}$ taking values in the $\infty $-category $\text{CAlg}_\Lambda ^\Delta $.  For a sheaf of animated $\Lambda $-algebras $A_\bullet $ we can consider, as in \cite[2.1.0.1]{LurieSAG}, the associated sheaf of $\mathbf{E}_\infty $-algebras and the associated module category, which we will denote $\text{Mod}_{(\mathbf{S}, A_\bullet )}$.  If $A_\bullet $ is a simplicial object in the category of sheaves of $\Lambda $-algebras then we will also denote by $\text{Mod}_{(\mathbf{S}, A_\bullet )}$ the module category of the associated sheaf of $\mathbf{E}_\infty $-algebras.

As discussed in appendix \ref{A:appendixA}, for a simplicial sheaf of $\Lambda $-algebras $A_\bullet $ we can also consider its normalized complex $N(A_\bullet )$, which is a sheaf of strictly commutative differential graded algebras, and its associated category of sheaves of differential graded modules $\text{Mod}^\dg _{(\mathbf{S}, N(A_\bullet ))},$ viewed as a model category with the flat model category structure.  This is a differential graded category and by the general construction of \cite[1.3.1.6]{LurieHA} we get an $\infty $-category 
$$
\mls D (\mathbf{S}, A_\bullet ):= N_\dg (\text{Mod}^{\dg , \circ }_{(\mathbf{S}, N(A_\bullet ))}), 
$$
where $\text{Mod}^{\dg , \circ }_{(\mathbf{S}, N(A_\bullet ))}\subset \text{Mod}^{\dg }_{(\mathbf{S}, N(A_\bullet ))}$ denotes the subcategory of fibrant-cofibrant objects.
We write $D(\BaseSite, A_\bullet )$ for the associated homotopy category.
As noted in \ref{A:thm1} (in the case of a discrete ring $A$ this is \cite[2.1.2.3]{LurieSAG}) the $\infty $-category $\mls D (\mathbf{S}, A_\bullet )$ is naturally identified with the hypercomplete objects in $\text{Mod}_{(\mathbf{S}, A_\bullet )}$.  For our purposes studying perfect complexes, this distinction between $\mls D (\mathbf{S}, A_\bullet )$ and $\text{Mod}_{(\mathbf{S}, A_\bullet )}$ is not important, and it is a matter of preference as to which $\infty $-categorical version of the derived category one works with.    
\end{pg}

\begin{rem}
For two objects $M, N \in \text{Mod}^{\dg , \circ }_{(\mathbf{S}, N(A_\bullet ))}$ defining objects of $\mls D(\BaseSite, A_\bullet )$, a description of the mapping space
$$
\mathrm{Map}_{\mls D(\BaseSite, A_\bullet )}(M, N)
$$
is provided by \cite[1.3.1.12]{LurieHA}, which shows that 
$$
\mathrm{Map}_{\mls D(\BaseSite, A_\bullet )}(M, N)\simeq \DK(\tau _{\leq 0}\mathrm{Hom}^\bullet _{N(A_\bullet )}(M, N)),
$$
where on the right we consider truncation of the mapping complex followed by the Dold-Kan functor (see for example \cite[1.2.3.5]{LurieHA}).
Note furthermore that since any object of $\text{Mod}^{\dg , \circ }_{(\mathbf{S}, N(A_\bullet ))}$ is fibrant the complex $\mathrm{Hom}^\bullet _{N(A_\bullet )}(M, N)$ is calculating the internal Hom-complex in the homotopy category $\mathrm{Ho}(\text{Mod}^{\dg }_{(\mathbf{S}, N(A_\bullet ))}).$
\end{rem}

\subsection{Deformations}

\begin{pg}
For a simplicial sheaf of $\Lambda $-algebras $A_\bullet $ we can also (by \ref{L:A.6}) describe $\mls D(\BaseSite, A_\bullet )$ as the $\infty $-category
$$
N_\dg (\text{Mod}_{(\BaseSite, N(A_\bullet ))}^{\dg , \cof})[W^{-1}],
$$
obtained by localizing the dg-nerve of cofibrant objects in $\text{Mod}_{(\BaseSite, N(A_\bullet ))}^{\dg}$ along weak equivalences.
\end{pg}

\begin{pg}
If $A_\bullet \rightarrow B_\bullet $ is a map of sheaves of simplicial $\Lambda $-algebras then there is an induced functor
$$
B_\bullet \lotimes _{A_\bullet }(-):\mls D(\BaseSite, A_\bullet )\rightarrow \mls D(\BaseSite, B_\bullet ).
$$
  This functor is induced by the tensor product
$$
N(B_\bullet )\otimes _{N(A_\bullet )}(-):\mathrm{Mod}_{(\BaseSite, N(A_\bullet ))}^\cof \rightarrow \mathrm{Mod}_{(\BaseSite , N(B_\bullet ))}^\cof 
$$
To prove that this gives a well-defined functor on localizations, we must show that if $a:M \rightarrow N $ is an equivalence in $\mathrm{Mod}_{(\BaseSite , N(B_\bullet ))}^\cof $ then 
$$
N(B_\bullet )\otimes _{N(A_\bullet ) }M \rightarrow N(B_\bullet ) \otimes _{N(A_\bullet ) }N
$$
is an equivalence in $\mathrm{Mod}_{(\BaseSite , N(B_\bullet ))}^\cof $.  For this it suffices to show that the adjoint forgetful functor
$$
\mathrm{Mod}_{(\BaseSite, N(B_\bullet ))}\rightarrow \mathrm{Mod}_{(\BaseSite ,N(A_\bullet ))}
$$
preserves fibrations and trivial fibrations -- this is immediate from the definitions.  Note also that this functor induces the usual derived tensor product on the homotopy categories.
\end{pg}

\begin{pg} Let $A'_\bullet \rightarrow A_\bullet $ be a surjective map of simplicial $\Lambda $-algebras with kernel $I_\bullet $ satisfying $I_\bullet ^2 = 0$.  Let $E \in \mls D(\BaseSite, A_\bullet )$ be an object.  We denote by $\text{\rm Def}_\infty (E)$ the homotopy fiber product of the diagram
\begin{equation}\label{E:diag0}
\xymatrix{
& \mls D(\BaseSite, A_\bullet ')\ar[d]^-{A_\bullet \lotimes _{A'_\bullet }(-)}\\
\star \ar[r]^-{E }& \mls D(\BaseSite, A_\bullet ).}
\end{equation}
\end{pg}

\begin{pg}\label{P:2.12}
Let $\Def (E )$ denote the category whose objects are pairs $(E^{\prime }, \sigma )$, where $E^{\prime }\in D(\BaseSite, A'_\bullet )$ is an object and 
$$
\sigma :E^{\prime }\lotimes _{N(A'_\bullet )}N(A_\bullet )\rightarrow E
$$
is an isomorphism in $D(\BaseSite, A_\bullet )$.  A morphism
\begin{equation}\label{E:defmorphism}
q :(E_1^{\prime }, \sigma _1)\rightarrow (E_2^{\prime }, \sigma _2)
\end{equation}
in $\Def (E )$ is given by a morphism $\rho :E_1^{\prime }\rightarrow E_2^{\prime }$ in $D(\BaseSite, A'_\bullet )$ such that the diagram in $D(\BaseSite, A_\bullet )$
$$
\xymatrix{
E_1^{\prime }\lotimes _{N(A'_\bullet )}N(A_\bullet )\ar[rr]^-\rho \ar[rd]_-{\sigma _1}&& E_2^{\prime }\ar[ld]^-{\sigma _2}\lotimes _{N(A'_\bullet )}N(A_\bullet )\\
& E& }
$$
commutes.

There is a natural map
\begin{equation}\label{E:homotopymap}
\text{Ho}(\text{\rm Def}_\infty (E))\rightarrow \Def (E ).
\end{equation}
Indeed the category $\Def (E )$ is the categorical fiber product of the diagram
\begin{equation}\label{E:diag1}
\xymatrix{
& \text{Ho}(\mls D(\BaseSite, A_\bullet '))\ar[d]^-{A_\bullet \lotimes _{A'_\bullet }(-)}\\
\star \ar[r]^-{E }& \text{Ho}(\mls D(\BaseSite, A_\bullet )).}
\end{equation}
Now by general adjunction properties of passing to the homotopy category \cite[1.2.3.1]{LurieHT}, the diagram \eqref{E:diag0} maps to the diagram obtained by applying the nerve to \eqref{E:diag1}.  By passing to the homotopy categories of the associated homotopy fibers we get the map \eqref{E:homotopymap}.

We can understand the $\infty $-category $\text{\rm Def}_\infty (E)$ and its relationship with $\Def (E )$ as follows.
\end{pg}

\begin{pg}
Note first of all that $\text{\rm Def}_\infty (E)$ is a groupoid in the sense of \cite[1.2.5]{LurieHT} (that is, its homotopy category is a groupoid).  This follows from observing that if $E^{\prime }\in \mls D(\BaseSite, A'_\bullet )$ is an object with an equivalence $\sigma :E^{\prime }\lotimes _{N(A'_\bullet )}N(A_\bullet )\rightarrow E$ then tensoring the sequence of $N(A_\bullet ')$-modules
$$
0\rightarrow N(I_\bullet )\rightarrow N(A'_\bullet )\rightarrow N(A_\bullet )\rightarrow 0
$$
with $E^{\prime }$ we get a distinguished triangle 
$$
\xymatrix{
E \lotimes _{N(A_\bullet )}N(I_\bullet )\ar[r]& E^{\prime }\ar[r]& E \ar[r]& E \lotimes _{N(A_\bullet )}N(I_\bullet )[1]}
$$
in the triangulated category $D(\BaseSite, A'_\bullet ):= \text{Ho}(\mls D(\BaseSite, A'_\bullet )).$

It follows that for a morphism of pairs (that is, a morphism in $\mls D(\BaseSite, A'_\bullet )$ compatible with the identifications with $E$ in $\mls D(\BaseSite, A_\bullet )$)
$$
\rho :(E_1^{\prime }, \sigma _1)\rightarrow (E_2^{\prime }, \sigma _2)
$$
we get an induced morphism of distinguished triangles in $D(\BaseSite, A'_\bullet )$
$$
\xymatrix{
E \lotimes _{N(A_\bullet )}N(I_\bullet )\ar[r]\ar@{=}[d]& E^{\prime }_1\ar[r]\ar[d]^-\rho & E \ar[r]\ar@{=}[d]& E \lotimes _{N(A_\bullet )}N(I_\bullet )[1]\ar@{=}[d]\\
E \lotimes _{N(A_\bullet )}N(I_\bullet )\ar[r]& E^{\prime }_2\ar[r]& E \ar[r]& E \lotimes _{N(A_\bullet )}N(I_\bullet )[1],}
$$
and therefore $\rho $ is an equivalence.

It follows that $\text{\rm Def}_\infty (E)$ can also be described as the fiber product of the underlying $\infty $-groupoids
\begin{equation}\label{E:diag2}
\xymatrix{& \mls D(\BaseSite, A'_\bullet )^\simeq \ar[d]^-{A_\bullet \lotimes _{A'_\bullet }(-)}\\
\star \ar[r]^-{E }& \mls D(\BaseSite, A_\bullet )^\simeq .}
\end{equation}

From this, and looking at the associated long exact sequences of homotopy groups associated to \eqref{E:diag1} and \eqref{E:diag2} one also gets that \eqref{E:defmorphism} induces a bijection on isomorphism classes of objects, so we can think of objects of $\text{\rm Def}_\infty (E)$ as pairs $(E^{\prime }, \rho )$ as above.
\end{pg}

\begin{rem} The above argument shows, in fact, that a morphism in $\mls D(\BaseSite, A'_\bullet )$ is an equivalence if and only if its image in $\mls D(\BaseSite, A_\bullet )$ is an equivalence.
\end{rem}

\begin{pg}
This remark implies that for 
an object $(E^{\prime }, \rho )\in \text{\rm Def}_\infty (E)$ the derived automorphism group
$$
\text{\rm Aut}^\infty (E^{\prime }, \rho ) = \Omega _{(E^{\prime }, \rho )}(\text{\rm Def}_\infty (E))
$$
can be described as the homotopy fiber over $0$ of the map
$$
\text{DK}(\tau _{\leq 0}\mathrm{RHom}_{N(A_\bullet ')}(E^{\prime }, E^{\prime }))\rightarrow \text{DK}(\tau _{\leq 0}\mathrm{RHom}_{N(A_\bullet )}(E^{\bullet }, E^{\bullet })).
$$
Applying $\mathrm{RHom}_{N(A'_\bullet )}(E^{\prime }, -)$ to the distinguished triangle
$$
E \lotimes _{N(A_\bullet )}N(I_\bullet )\rightarrow E^{\prime }\rightarrow E \rightarrow E \lotimes _{N(A_\bullet )}N(I_\bullet )[1]
$$
we see that 
$$
\text{\rm Aut}^\infty (E^{\prime }, \rho )\simeq \text{DK}(\tau _{\leq 0}\mathrm{RHom}_{N(A_\bullet )}(E , E \lotimes N(I_\bullet ))).
$$
From this we conclude that the $\infty $-groupoid $\text{\rm Def}_\infty (E)$ is equivalent to
\begin{equation}\label{E:step1}
\text{DK}(\tau _{\leq 0}\mathrm{RHom}_{N(A_\bullet )}(E , E \lotimes N(I_\bullet ))[1])\times \pi _0(\Def (E )).
\end{equation}
\end{pg}

\begin{pg}
In the case when $\text{\rm Def}_\infty (E)$ is nonempty we can describe the set of isomorphism classes as follows.  Fix one lifting $(E_0^{\prime }, \rho _0)\in \Def (E )$, and define a map
\begin{equation}\label{E:2.14.1}
\pi _0(\Def (E ))\rightarrow \mathrm{Ext}^1_{N(A_\bullet )}(E, E \lotimes N(I_\bullet ))
\end{equation}
by sending an object $(E^{\prime }, \rho )$ to the image of $\rho ^{-1}\circ \rho _0 $ under the map
$$
\mathrm{Hom}_{N(A_\bullet )}(E_0'\lotimes N(A_\bullet ), E^{\prime }\lotimes N(A_\bullet ) )\simeq \mathrm{Hom}_{N(A'_\bullet )}(E_0^{\prime }, E^{\prime }\lotimes N(A_\bullet ))\rightarrow \mathrm{Ext}^1_{N(A_\bullet )}(E, E \lotimes N(I_\bullet ))
$$
obtained from the distinguished triangle
$$
E \lotimes N(I_\bullet )\rightarrow E^{\prime }\rightarrow E \rightarrow E \lotimes N(I_\bullet )[1]
$$
by applying 
$$
\mathrm{Hom}_{N(A'_\bullet )}(E_0^{\prime }, -).
$$
The image of the class of $(E^{\prime }, \rho )$ under \eqref{E:2.14.1} is by construction zero if and only if the morphism $\rho  _0^{ -1}\circ \iota $ lifts to $N(A'_\bullet )$, which implies that \eqref{E:2.14.1} is injective.  

The map is also surjective.  This can be seen as follows.    For a class
$$
\alpha \in \mathrm{Ext}^1_{N(A_\bullet )}(E, E \lotimes N(I_\bullet ))
$$
we can represent $\alpha $ by a map of $N(A_\bullet )$-modules
$$
\tilde \alpha :E^{\prime \bullet}_0\otimes _{N(A'_\bullet ) }N(A_\bullet )\rightarrow E \otimes N(I_\bullet )[1].
$$
Viewing this as a map of $N(A'_\bullet ) $-modules and taking cones we get 
 a short exact sequence of $N(A'_\bullet ) $-modules
$$
0\rightarrow E \otimes N(I_\bullet )\rightarrow T_\alpha  \rightarrow E^{\prime }_0\otimes _{N(A'_\bullet ) }N(A_\bullet )\rightarrow 0.
$$
Taking direct sum with $E^{\prime }_0$ we get a short exact sequence
$$
0\rightarrow (E \otimes N(I_\bullet ))^{\oplus 2}\rightarrow T_\alpha  \oplus E^{\prime }_0 \rightarrow E^{\prime }_0 \otimes N(A_\bullet ) \oplus E \rightarrow 0.
$$
Pulling this back along the graph 
$$
E^{\prime }_0 \otimes N(A_\bullet )\rightarrow E^{\prime }_0 \otimes N(A_\bullet )\oplus E 
$$
of $\rho $ and pushing out along the summation map
$$
(E \otimes N(I_\bullet ))^{\oplus 2}\rightarrow E \otimes N(I_\bullet )
$$
we get an extension of $N(A'_\bullet ) $-modules
$$
0\rightarrow E \otimes N(I_\bullet )\rightarrow E_\alpha ^{\prime }\rightarrow E^{\prime }_0 \otimes N(A_\bullet )\rightarrow 0.
$$
We leave it to the reader to check that this defines an object of $\Def (E )$ with class $\alpha $.
\end{pg}

\begin{pg}\label{P:2.15}
Combining this with \eqref{E:step1} we find that in the case when $\text{\rm Def}_\infty (E)$ is nonempty the pullback of \eqref{E:diag0} can be described as
$$
\DK (\tau _{\leq 0}(\mathrm{Hom}_{N(A_\bullet )}(E, E \lotimes N(I_\bullet ))[1])).
$$
\end{pg}
\begin{pg}\label{P:2.16}
In what follows we will also consider the subcategory 
$$
\mls D_{\perf}(\BaseSite, A_\bullet )\subset \mls D (\BaseSite, A_\bullet )
$$
of perfect $A_\bullet $-modules \cite[7.2.4.1]{LurieHA}.  This is again a stable $\infty $-category.

We will consider this as a symmetric monoidal stable $\infty $-category with the monoidal structure given by direct sums.

In the case when the site $\BaseSite$ is trivial (e.g. one object and one morphism) we write simply $\mls D(A_\bullet )$ (resp. $\mls D_\perf (A_\bullet )$) for $\mls D(\BaseSite, A_\bullet )$ (resp. $\mls D_\perf (\BaseSite, A_\bullet )$).
\end{pg}

\begin{rem} 
In many cases the notion of perfect complex coincides with the notion of dualizable object but we do not know the relationship between the two notions in general.\footnote{This question was earlier asked by Daniel Bergh \url{https://mathoverflow.net/questions/313318/are-dualizable-objects-in-the-derived-category-of-a-ringed-topos-perfect}}
\end{rem}




\section{Various descriptions of $K$-theory}\label{S:Ktheory}

In this section we summarize for the convenience of the reader a few basic approaches to  and results about algebraic $K$-theory that we will need in what follows.

\subsection{$K$-theory as group completion}

In this approach to $K$-theory one starts with the $\infty $-category of $\mathbf{E}_\infty $-monoids and its subcategory of grouplike  $\mathbf{E}_\infty $-monoids \cite[5.2.6.6]{LurieHA}.  By \cite[5.2.6.26]{LurieHA}, this  $\infty $-subcategory is equivalent to the $\infty $-category of connective spectra $\mathrm{Sp}^{\geq 0}$.  We can then consider the group completion functor
$$
\xymatrix{
(\text{$\mathbf{E}_\infty $-monoids})\rightarrow (\text{grouplike $\mathbf{E}_\infty $-monoids})}\simeq \mathrm{Sp}^{\geq 0}.
$$
If $\mls P$ is a symmetric monoidal category the nerve of the underlying groupoid $\mls P^{\simeq }$ is an $\mathbf{E}_\infty $-monoid and the $K$-theory of $\mls P$, denoted $K(\mls P)$, is defined as the associated group completion.

\begin{example}\label{E:aring} For a ring $R$ the category $\mathrm{Proj}(R)$ of projective $R$-modules is symmetric monoidal under $\oplus $ and $K(R)$ is defined to be the group completion of the nerve of the underlying groupoid of $\mathrm{Proj}(R)$.
\end{example}

\subsection{$K$-theory as universal additive invariant}

The main reference for this approach is \cite{BGT}. 
Let $\mathrm{Cat}_\infty ^{\mathrm{ex}}$ denote the $\infty $-category of  small, idempotent complete, stable $\infty $-categories, with morphisms given by exact functors (see \cite[2.12]{BGT}).    The main result of \cite{BGT} is then that there is a universal ``additive'' invariant 
$$
U:\mathrm{Cat}_\infty ^{\mathrm{ex}}\rightarrow \mls M_{\mathrm{add}},
$$
where the target is again a presentable stable $\infty $-category.  In fact, $\mls M_{\mathrm{add}}$ is monoidal with unit object $\mathbf{1}_{\mls M_{\mathrm{add}}}$ given by applying $U$ to the compact objects in the stable $\infty $-category of spectra.  Given an object $\mls D\in \mathrm{Cat}_\infty ^{\mathrm{ex}}$ we can form the mapping spectrum (see \cite[2.15]{BGT})
$$
\mathrm{Map}(\mathbf{1}_{\mls M_{\mathrm{add}}}, U(\mls D)).
$$
By \cite[7.13]{BGT}  this defines connective algebraic $K$-theory
$$
K(\mls D):= \mathrm{Map}(\mathbf{1}_{\mls M_{\mathrm{add}}},U(\mls  D)).
$$
Because the functor $U$ is monoidal we have an induced map of spectra
$$
\mathrm{Map}_{\mathrm{Cat}_\infty ^{\mathrm{ex}}}(\mathbf{1}_{\mathrm{Cat}_\infty ^{\mathrm{ex}}}, \mls D)\rightarrow \mathrm{Map}(\mathbf{1}_{\mls M_{\mathrm{add}}},U(\mls  D)).
$$
This induces  a map
$$
\mls D^{\simeq }\rightarrow K(\mls D),
$$
from the underlying $\infty $-groupoid $\mls D^{\simeq }$ of $\mls D$.

One can also describe the algebraic $K$-theory of a small, idempotent complete, stable $\infty $-category using an appropriate version of the Waldhausen construction.  This is discussed in \cite[\S 7.1 and \S 7.2]{BGT}.  See also \cite{Barwick}.  This explicit construction makes the functoriality of $K$-theory clear, and in particular enables us to consider presheaves of small, idempotent complete, stable $\infty $-categories and their associated $K$-theory.

\begin{example} If $A_\bullet $ is a simplicial ring then $K(A_\bullet )$ is defined to be the $K$-theory of the stable $\infty $-category $\mls D_\perf (A_\bullet )$ defined in \cref{P:2.16}.

If $A_\bullet = R$ is a ring, then this recovers the $K$-theory defined in \cref{E:aring}.  Namely, the inclusion $\mathrm{Proj}(R)^{\simeq }\hookrightarrow \mls D_\perf (R)^{\simeq }$ induces a monoidal map (where the target is defined using the definition in \cite{BGT})
$$
\mathrm{Proj}(R)^{\simeq }\rightarrow K(\mls D_\perf (R)).
$$
By the universal property of group completion this induces a map (where the left side is defined using group completion)
$$
K(R)\rightarrow K(\mls D_\perf (R)).
$$
That this map is an equivalence follows, for example, by comparison with Waldhausen $K$-theory.
\end{example}

\subsection{$K$-theory and Picard groupoids}\label{P:3.5}

The relationship between $K$-theory and  determinants in the setting of functors to Picard categories as developed in \cite{Knudsen} is discussed in \cite{MTW}.

The setting here is that of a Waldhausen category $\mls W$ \cite[1.2]{Waldhausen}.  Let $w(\mls W)$ denote the category with the same objects as $\mls W$ but morphisms the weak equivalences of $\mls W$ and let $\mathrm{cof}(\mls W)$ denote the category whose objects are cofiber sequences
$$
\xymatrix{
A\ar@{^{(}->}[r]& B\ar@{->>}[r]& C}
$$
and whose morphisms are commutative diagrams
$$
\xymatrix{
A_1\ar@{^{(}->}[r]\ar[d]& B_1\ar[d]\ar@{->>}[r]& C_1\ar[d]\\
A_2\ar@{^{(}->}[r]& B_2\ar@{->>}[r]& C_2,
}
$$
where the vertical morphisms are weak equivalences.  For a commutative Picard category $\mls P$ (see for example \cite[\S 4]{Delignedeterminant}) a notion of determinant functor from $\mls W$ to $\mls P$ is defined in \cite[ \S 1.2]{MTW}.  This is a functor
$$
\delta :w(\mls W)\rightarrow \mls P
$$
together with an isomorphism $\sigma $ between the two induced functors
$$
\mathrm{cof}(\mls W)\rightarrow \mls P
$$
given by
$$
\xymatrix{
(A\ar@{^{(}->}[r]& B\ar@{->>}[r]& C)}\mapsto \delta (B)
$$
and
$$
\xymatrix{
(A\ar@{^{(}->}[r]& B\ar@{->>}[r]& C)}\mapsto \delta (A)+_{\mls P}\delta (C).
$$
The data $(\delta, \sigma )$ is required to satisfy various natural compatibilities detailed in  \cite[1.2.3]{MTW}.

A Picard category $\mls P$ defines a grouplike $\mathbf{E}_\infty $-monoid.  This is explained in \cite[12.5 and 12.15]{BS}.  By \cite[5.3]{Patel} this defines an equivalence between homotopy categories of Picard categories and $1$-truncated connective spectra.  

On the other hand, we can consider the $1$-truncation $\tau _{\leq 1}K(\mls W)$ of the Waldhausen $K$-theory of $\mls W$, which comes with a map
$$
\delta :w(\mls W)\rightarrow \tau _{\leq 1}K(\mls W).
$$
It is shown in \cite[1.6.3]{MTW}    that this functor has the structure of a universal determinant functor.  

The relationship between this approach and the perspective on $K$-theory as a universal additive invariant is discussed in \cite[\S 7.2]{BGT}.  

We do not need to develop the full theory here.  In order to have the appropriate functoriality, however, it is important to note that the comparison map is induced by an explicit map of $\infty $-categories in our context.  Namely, let $A$ be a ring and let $\mathrm{Perf}(A)$ denote the Waldhausen category of perfect complexes of $A$-modules, which is also a dg-category.  Then there is an induced functor of $\infty $-categories
$$
N(\mathrm{Perf}(A))\rightarrow N_{\mathrm{dg}}(\mathrm{Perf}(A))
$$
and the comparison between the two approaches is obtained by then applying the $\infty $-categorical version of the Waldhausen $S$-construction.  In particular, the comparison map is functorial in $A$.

\begin{rem} 
  We  have an inclusion $\mathrm{Proj}(A)\subset \mathrm{Perf}(A)$ of the category of projective modules, which induces an isomorphism on $K$-theory.  It follows from this that restriction defines an equivalence of categories between the category of determinant functors on $\mathrm{Perf}(A)$ and determinant functors on $\mathrm{Proj}(A)$.
\end{rem}

\subsection{$K$-theory and left Kan extension}\label{P:3.7}

The $K$-theory of animated rings can be described using the $K$-theory of ordinary rings as a Kan extension from smooth $\mathbf{Z}$-algebras. This result is due to Bhatt and Lurie (see \cite[A.0.6]{EHKSY}).

Let $\mathrm{Alg}_{\mathbf{Z}}$ denote the category of commutative rings.  Algebraic $K$-theory defines a functor
$$
K:\mathrm{Alg}_{\mathbf{Z}}\rightarrow \mathrm{Sp}^{\geq 0}.
$$
We can then consider the left Kan extension of this functor to get a functor
$$
\mathrm{CAlg}_{\mathbf{Z}}^\Delta \rightarrow \mathrm{Sp}^{\geq 0}
$$
from animated rings to spaces.
The result of Bhatt and Lurie states that this gives $K$-theory of animated rings, defined as a universal additive invariant.  

\section{Determinants: Punctual Case}\label{S:section3}

We are ultimately interested in studying determinants and traces for complexes on a general site, but as a first step in that direction we develop some preliminary material in this section in the case of the punctual site.

\begin{pg}
Let $A_\bullet $ be a simplicial ring with associated $K$-theory $K(A_\bullet )$.

Define $GL_n(A_\bullet )$ to be the fiber product of simplicial monoids of the diagram
$$
\xymatrix{
& M_n(A_\bullet )\ar[d]\\
GL_n(\pi _0(A_\bullet ))\ar[r]& M_n(\pi _0(A_\bullet ))}
$$
and let $\widehat {GL}(A_\bullet )$ denote the colimit of the $GL_n(A_\bullet )$ along the standard inclusions
$$
GL_n(A_\bullet )\hookrightarrow GL_{n+1}(A_\bullet ).
$$

The free modules $A_\bullet ^{\oplus n}$ define a map of $\mathbf{E}_\infty $-monoids
$$
\coprod _{n\geq 0}BGL_n(A_\bullet )\rightarrow \Perf ({A_\bullet })^\simeq .
$$
Here the monoidal structure on the left is induced by the natural isomorphisms
$$
A_\bullet ^{\oplus n}\oplus A_\bullet ^{\oplus m}\simeq A^{\oplus (n+m)}.
$$
\end{pg}
\begin{lem}\label{L:2.3}
The induced map on $K$-theory
$$
K(\coprod _{n\geq 0}BGL_n(A_\bullet ))\rightarrow K(A_\bullet )
$$
induces an isomorphism after Zariski sheafification on $\Sp (\pi _0(A_\bullet ))$.
\end{lem}

\begin{rem} The notion of a sheaf taking values in the $\infty $-category of spaces is introduced in \cite[6.2.2.6]{LurieHT}.  In \cite[6.2.2.7]{LurieHT} it is shown that the $\infty $-category of sheaves can be viewed as a localization  of the $\infty $-category of presheaves, which implies that the inclusion of sheaves into presheaves has a left adjoint -- this is what we refer to as \emph{sheafification}. 
\end{rem}

\begin{proof}
As in \cite[9.39]{BGT} the infinite loop space associated to 
 $K(\coprod _nBGL_n(A_\bullet ))$ is isomorphic to
$$
\mathbf{Z}\times B\widehat {GL}(A_\bullet )^+.
$$
where $B\widehat {GL}(A_\bullet )^+$ denotes Quillen's plus construction.  In particular, we get a map
$$
\mathbf{Z}\times B\widehat {GL}(A_\bullet )^+\rightarrow \mathbf{Z}\times _{K_0(A_\bullet )}K(A_\bullet ).
$$
By \cite[9.40]{BGT} this map is an equivalence.  To prove the lemma it therefore suffices to observe that every object of $K_0(A_\bullet ) = K_0(\pi _0(A_\bullet ))$ (isomorphism given as in \cite[2.16]{KST}) is locally in the image of $\mathbf{Z}$, which is immediate.
\end{proof}

\begin{rem} Note that the functors $GL_n(-)$ and $\widehat {GL}(-)$ reflect weak equivalence and therefore induce functors
$$
\text{CAlg}_{\mathbf{Z}}^\Delta \rightarrow (\mathbf{E}_\infty -\text{monoids}).
$$
By \cite[A.0.4]{EHKSY} these functors are equal to the Kan extensions of their restrictions to smooth algebras.
\end{rem}

 

\begin{pg} 
Let $\mls Pic^{\mathbf{Z}}(A_\bullet )$ denote the Zariski sheafification of the grouplike $\mathbf{E}_\infty $-monoid given by 
$$
\mathbf{Z}\times BGL_1(A_\bullet )
$$
with monoidal structure given by addition on $\mathbf{Z}$ and multiplication on $BGL_1(A_\bullet )$ but with signed commutativity constraint as in \cite[12.2 (iii)]{BS}.

There is a projection map
$$
\mls Pic^{\mathbf{Z}}(A_\bullet )\rightarrow \mathbf{Z}
$$
with fiber $BGL_1(A_\bullet )$.

Here $\mathbf{Z}$ should be understood as the global sections of the Zariski sheaf associated to the constant sheaf on $\Sp (\pi _0(A_\bullet ))$.

Note that with this definition $\mls Pic^{\mathbf{Z}}(A_\bullet )$ forms a Zariski sheaf on $\Sp (\pi _0(A_\bullet ))$.  One can also describe $\mls Pic ^{\mathbf{Z}}(A_\bullet )$ as the Picard spectrum associated to the symmetric monoidal $\infty $-category of finitely generated projective $A_\bullet $-modules (see \cref{R:6.4} below). Note also that $\mls Pic ^{\mathbf{Z}}(-)$ extends to a functor on $\text{CAlg}_{\mathbf{Z}}^\Delta$.
\end{pg}

\begin{pg}
The determinant map on $GL_n(A_\bullet )$ defines a map of symmetric monoidal $\infty $-groupoids
$$
\det :\coprod _{n\geq 0}BGL_n(A_\bullet )\rightarrow \mls Pic^{\mathbf{Z}}(A_\bullet ).
$$
For $A_\bullet $ an ordinary ring this is immediate, and since both sides are left Kan extensions of their restrictions to smooth algebras we get the map also for simplicial rings.
Since the target is grouplike this induces a map
$$
K(\coprod _{n\geq 0}BGL_n(A_\bullet ))\rightarrow \mls Pic^{\mathbf{Z}}(A_\bullet ).
$$
Passing to the associated sheafifications we get a map
$$
\det :K(A_\bullet )\rightarrow \mls Pic^{\mathbf{Z}}(A_\bullet ).
$$
We call this map, as well as the corresponding map
\begin{equation}\label{E:4.6.1}
\det :\mls D_\perf (A_\bullet )^\simeq \rightarrow \mls Pic^{\mathbf{Z}}(A_\bullet )
\end{equation}
the \emph{determinant map}.
\end{pg}

\section{Perfect complexes on ringed sites}\label{S:section4}

\begin{pg} Let $(\BaseSite , \mls O)$ be a ringed site.  For a simplicial $\mls O$-algebra $A_\bullet $ we have the corresponding stable $\infty $-category $\mls D(\BaseSite , A_\bullet )$.  Let 
$$
\mls D_{\perf}^{\text{\rm strict}}(\BaseSite, A_\bullet )\subset \mls D(\BaseSite, A_\bullet )
$$
denote the smallest stable $\infty $-subcategory containing $A_\bullet $ and which is closed under retracts, and define
$$
\mls D_\perf (\BaseSite, A_\bullet )\subset \mls D(\BaseSite, A_\bullet )
$$
to be the $\infty $-subcategory of objects $M$ for which there exists a collection of objects $\{U_i\}_{i\in I}$ covering the final object of the topos associated to $\BaseSite $ such that the restriction of $M$ to each $U_i$ lies in $\mls D^{\text{\rm strict}}_{\perf }(\BaseSite |_{U_i}, A_\bullet |_{U_i})\subset \mls D (\BaseSite |_{U_i}, A_\bullet |_{U_i}).$  So we have
$$
\mls D^{\text{strict}}_{\perf}(\BaseSite, A_\bullet )\subset \mls D_{\perf }(\BaseSite, A_\bullet )\subset \mls D(\BaseSite, A_\bullet ).
$$
\end{pg}

\begin{pg}
Let 
\begin{equation}\label{E:thecats} 
\underline {\mls D}_\perf ^{\mathrm{strict}} \ \ (\text{resp.} \ \uPerf, \uPerf ^\prime )
\end{equation}
be the presheaf of symmetric monoidal $\infty $-categories which to any $U\in \BaseSite $ associates 
$$
\mls D^{\text{strict}}_\perf (\BaseSite |_U, A_{\bullet }|_U), \ \ (\text{resp.} \ \mls D_\perf (\BaseSite |_U, A_{\bullet }|_U), \mls D_\perf (A_\bullet (U)).
$$
There are natural maps
\begin{equation}\label{E:5.2.1}
\uPerf ^\prime \rightarrow \underline {\mls D}_\perf ^{\text{strict}}\rightarrow \uPerf .
\end{equation}
\end{pg}

\begin{rem} 
In the case when $A_\bullet = \mls O$ is a sheaf of (ordinary) rings,  the sheaf $\uPerf $ is equivalent to the sheaf that associates to any $U$ the $\infty $-category of perfect complexes of $\mls O$-modules on $U$ in the sense of \cite[Tag 08FL]{stacks-project}.    
\end{rem}

\begin{lem}\label{L:5.4b}  The presheaf $\uPerf $ is a sheaf, and both the maps
$$
\uPerf ^\prime \rightarrow \uPerf, \ \ \underline {\mls D}_\perf ^{\text{\rm strict}}\rightarrow \uPerf 
$$
induce equivalences upon sheafification.
\end{lem}
\begin{proof}
The statements that $\uPerf $ is a sheaf and that 
$$
\underline {\mls D}_\perf ^{\text{strict}}\rightarrow \uPerf
$$
induces an equivalence upon sheafification follows immediately from the definition of $\uPerf $.

To prove the lemma we therefore show that the map
$$
\uPerf ^\prime \rightarrow \uPerf
$$
induces an equivalence upon sheafification.  Let $\uPerf ^{\prime a}$ denote the sheaf associated to $\uPerf ^{\prime }$.

Since sheafification commutes with finite limits \cite[6.2.2.7]{LurieHT}, for two objects $x, y\in \uPerf ^\prime (U)$ with associated objects $x^a, y^a\in \uPerf ^{\prime a}$ and  $F_x, F_y\in \uPerf (U)$, the sheaf 
$$
\underline {\text{Map}}_{\uPerf ^{\prime a}}(x^a, y^a)
$$
is equal to the sheaf associated to the presheaf 
\begin{equation}\label{E:thepresheaf}
V\mapsto \text{Map}_{\mls D_\perf (A_\bullet (V))}(x\otimes _{A_\bullet (U)}A_\bullet (V), y\otimes _{A_\bullet (U)}A_\bullet (V)).
\end{equation}
Let $\mls S$ denote the subcategory of $\uPerf ^{\prime }(U)$ of those objects $x\in \uPerf ^{\prime }(U)$ for which the map
$$
\underline {\text{Map}}_{\uPerf ^{\prime a}}(x^a, y^a)\rightarrow \underline {\text{Map}}_{\uPerf }(F_x, F_y)
$$
is an equivalence for all $y\in \uPerf ^{\prime }(U)$.  Then $\mls S$ is a stable subcategory closed under retracts, so to show that $\mls S$ is equal to all of $\uPerf ^{\prime }(U)$ it suffices to show that $A_\bullet (U)$ is in $\mls S$.  

For $y\in \uPerf ^{\prime }(U)$ with associated object $F_y\in \uPerf (U)$ and $x = A_\bullet (U)$ the  presheaf \eqref{E:thepresheaf} is given by sending $V$ to 
$$
\tau _{\leq 0}(y\otimes _{A_\bullet (U)}A_\bullet (V)),
$$
where $y\otimes _{A_\bullet (U)}A_\bullet (V)$ is viewed as an object of $\mls D(\text{Ab})$ (the derived $\infty $-category of abelian groups).  On the other hand, the sheaf
$$
\underline {\text{Map}}_{\uPerf }(A_\bullet ^a, F_y)
$$
is given by applying $\tau _{\leq 0}$ to the sheafification of 
$$
V\mapsto y\otimes _{A_\bullet (U)}A_\bullet (V).
$$
Thus the statement that $A_\bullet (U) \in \mls S$ amounts to the observation that sheafification commutes with the functor $\tau _{\leq 0}$.    We conclude that $\mls S = \uPerf ^{\prime }(U).$

It follows that 
$$
\uPerf ^{\prime a}\rightarrow \uPerf 
$$
induces an equivalence on mapping spaces, and since locally every object is evidently in the image we conclude that this map is an equivalence.
\end{proof}

\begin{pg}\label{P:4.6}
The $\infty $-category $\mls D_\perf (\BaseSite, A_\bullet )$ is given by the global sections
$$
\Gamma (\BaseSite, \uPerf ).
$$
Let $\mls Pic ^{\mathbf{Z}}_{(\BaseSite, A_\bullet )}$ denote the global sections of the sheaf associated to the presheaf $\underline {\mls Pic}^{\mathbf{Z}}_{A_\bullet }$ given by
$$
U\mapsto \mls Pic ^{\mathbf{Z}}(A_\bullet (U)).
$$
Using \eqref{E:4.6.1} we then obtain a diagram 
$$
\xymatrix{
\uPerf ^{\prime, \simeq }\ar[r]\ar[d]^-{\det }& \uPerf ^\simeq \\
\underline {\mls Pic}_{\mls O}^{\mathbf{Z}}.&}
$$
Passing to the associated sheaves and using \cref{L:5.4b} we get an induced map 
\begin{equation}\label{E:3.6.1}
\det :\Perf {(\BaseSite, A_\bullet )}^\simeq \rightarrow \mls Pic^{\mathbf{Z}}_{(\BaseSite, A_\bullet )}.
\end{equation}
By functoriality of $K$-theory this map factors through the $K$-theory of $\Perf {(\BaseSite, A_\bullet )}$, which we denote by $K(\BaseSite, A_\bullet )$.
\end{pg}
\begin{pg} In the case when $A_\bullet = \mls O$ we can describe $\mls Pic ^{\mathbf{Z}}_{(\BaseSite, \mls O)}$ more explicitly as follows.  Let $\mathbf{Z}_{(\BaseSite, \mls O)}$ be the sheaf associated to the presheaf which to any $U\in \BaseSite$ associated the set of locally constant $\mathbf{Z}$-valued functions on $\Sp (\mls O(U))$.   Note that for any section $r\in \mathbf{Z}_{(\BaseSite, \mls O)}(U)$ the expression $$
(-1)^r\in \mls O^*(U)
$$
makes sense.  Then the sheaf $\underline {\mls Pic}_{(\BaseSite, \mls O)}^{\mathbf{Z}}$ associated to the presheaf $\underline {\mls Pic}^{\mathbf{Z}}_{\mls O}$ can be described as the stack in groupoids which to any $U$ associates the groupoid of pairs $(r, \mls L)$, where $r\in \mathbf{Z}_{(\BaseSite, \mls O)}(U)$ and $\mls L$ is an invertible module on $(\BaseSite|_U, \mls O)$, in the sense of \cite[Tag 0408]{stacks-project}.  The monoidal structure is given by
$$
(r, \mls L)* (r', \mls L'):= (r+r', \mls L\otimes \mls L')
$$
and the commutativity constraint
$$
(r, \mls L)*(r', \mls L')\simeq (r', \mls L')*(r, \mls L)
$$
is given by the isomorphism
$$
\mls L\otimes \mls L'\simeq \mls L'\otimes \mls L
$$
obtained by multiplying the isomorphism switching the factors with $(-1)^{rr'}.$

In particular, for any perfect complex $E$ on $(\BaseSite, \mls O)$ we can speak about its determinant $\det (E)$, an invertible $\mls O$-module.  
\end{pg}

\begin{pg}\label{P:3.8}
It will be useful to have a variant description of $\uPerf '$.

Let $\BaseSite^\zar $ denote the category whose objects are pairs $(U, V)$, where $U\in \BaseSite$ and $V\subset \Sp (\mls O(U))$ is an affine open set.  A morphism 
$$
(U', V')\rightarrow (U, V)
$$
is defined to be a morphism $f:U'\rightarrow U$ in $\BaseSite$ such that the induced morphism
$$
\Sp (\mls O(U'))\rightarrow \Sp (\mls O(U))
$$
sends $V'$ to $V$.  The \emph{Zariski topology} on $\BaseSite^\zar $ is defined by declaring a collection of morphisms
$$
\{f_i:(U_i, V_i)\rightarrow (U, V)\}
$$
a covering if each $f_i:U_i\rightarrow U$ is an isomorphism, and the collection of maps
$$
\{V_i\rightarrow V\}
$$
is an open covering of $V$.  With this definition the category of sheaves on $\BaseSite^\zar $ is equivalent to the category of collections of sheaves $\{F_U\}_{U\in \BaseSite}$, where $F_U$ is a sheaf on $\Sp (\mls O(U))$, together with transition morphisms $\theta _f:f^{-1}F_{U}\rightarrow F_{U'}$ for each morphism $f:U'\rightarrow U$ in $\BaseSite$, satisfying the natural cocycle condition.

A presheaf $F$ (of sets, $\infty $-categories etc.) on $\BaseSite^\zar $ induces a presheaf on $\BaseSite$ by composing $F$ with the functor
$$
\BaseSite\rightarrow \BaseSite^\zar , \ \ U\mapsto (U, \Sp (\mls O(U))).
$$
We denote this presheaf on $\BaseSite$ by $\gamma (F)$.

The presheaves of symmetric monoidal $\infty $-categories on $\BaseSite$
$$
\uPerf ', \ \ \underline {\mls Pic}^{\mathbf{Z}}_{\mls O}
$$
then extend to presheaves
$$
\uPerf ^{\prime \zar }, \ \ \underline {\mls Pic}^{\mathbf{Z}, \zar }_{\mls O}
$$
on $\BaseSite^\zar $, where $\uPerf ^{\prime \zar }$ sends $(U, V)$ to the category of strictly perfect complexes of $\mls O_{\Sp (\mls O(U))}(V)$-modules and $\underline {\mls Pic}^{\mathbf{Z}, \zar }_{\mls O}$ sends $(U, V)$ to the groupoid of $\mathbf{Z}$-graded line bundles on the scheme $V$.   Observe that we have
$$
\uPerf ' = \gamma (\uPerf ^{\prime , \zar }), \  \ \underline {\mls Pic}^{\mathbf{Z}}_{\mls O} = \gamma (\underline {\mls Pic}^{\mathbf{Z}, \zar }_{\mls O}).
$$

Note also that $\mls O$ extends to a presheaf of rings on $\BaseSite^\zar $, which we will denote by $\mls O^\zar $, given by
$$
\mls O^\zar (U, V) = \Gamma (V, \mls O_{\Sp (\mls O(U))}(V)).
$$

Similarly, any complex $I^\bullet $ of presheaves of $\mls O$-modules extends to a complex of $\mls O^\zar $-modules, which we will denote by $I^{\zar , \bullet }$, given by
$$
I^{\zar, \bullet }(U, V) = \widetilde {I (U)}(V),
$$
where $\widetilde {I (U)}$ denotes the complex of quasi-coherent sheaves on $\Sp (\mls O(U))$ associated to the complex of $\mls O(U)$-modules $I (U)$.

Define $BGL_n(\mls O^\zar )$ to be the presheaf on $\BaseSite^\zar $ which to any $(U, V)$ associates $BGL_n(\mls O^\zar (U, V))$.  We then have a natural map
$$
\coprod _nBGL_n(\mls O^\zar )\rightarrow \uPerf ^{\prime , \zar , \simeq }
$$
of presheaves of symmetric monoidal $\infty $-categories.  The determinant map also extends to a map
$$
\mdet ^\zar :\uPerf ^{\prime ,\zar , \simeq }\rightarrow \underline {\mls Pic}^{\mathbf{Z}, \zar }_{\mls O}
$$
inducing the previously defined determinant map after applying $\gamma $.  The advantage of working with presheaves on $\BaseSite^\zar $ is that by \cref{L:2.3} the map $\mdet ^\zar $, and hence also the determinant map \eqref{E:3.6.1}, is determined by the induced map
$$
\coprod _nBGL_n(\mls O^\zar )\rightarrow \underline {\mls Pic}^{\mathbf{Z}, \zar }_{\mls O}.
$$

More generally for a complex $I^\bullet $ of presheaves of $\mls O$-modules we can consider $\mls O^\zar [I^{\zar , \bullet }]$ on $\mathbf{S}^\zar $ and the determinant defines a map
$$
\mdet ^\zar :\underline {\mls D}_{\text{\rm perf}, \mls O[I^\bullet ]}^{\prime ,\zar , \simeq }\rightarrow \underline {\mls Pic}^{\mathbf{Z}, \zar }_{\mls O[I^\bullet ]}.
$$
\end{pg}

\section{Ring structure}\label{S:section5}

As pointed out to us by Bhargav Bhatt, the results of the previous section can profitably be upgraded to include statements about the ring structure on algebraic $K$-theory.  The following is a modification of an argument communicated to us by Bhatt, which answers, in particular, a question of R\"ossler\footnote{\url{https://mathoverflow.net/questions/354214/determinantal-identities-for-perfect-complexes}} which was also discussed in the Stacks Project\footnote{\url{https://www.math.columbia.edu/~dejong/wordpress/?p=4474}}.  The results of this section will not be used in what follows.

\begin{pg}
The ring structure on algebraic $K$-theory can be described in a few ways.  Most convenient for us is the description in \cite[8.6]{GGN} (see also \cite{BGT}).  This result is obtained from \cite[5.1]{GGN} which gives that the group completion functor
$$
\xymatrix{
(\text{$\mathbf{E}_\infty $-monoids})\rightarrow (\text{grouplike $\mathbf{E}_\infty $-monoids})}\simeq \mathrm{Sp}^{\geq 0}
$$
extends to a  symmetric monoidal functor.  Here the monoidal structure on $\mathrm{Sp}^{\geq 0}$ is given by the smashproduct of spectra (see \cite[5.3 (ii)]{GGN}). 

For a ringed site $(\BaseSite, \mls O)$ the underlying groupoids of  the objects \eqref{E:thecats} have the structure of presheaves of $\mathbf{E}_\infty $-semirings (see \cite[page 2, (iii)]{GGN}), and the diagram \eqref{E:5.2.1}
is compatible with this structure.

This implies, in particular, that the $K$-theory $K(\BaseSite, \mls O)$ has the structure of an $\mathbf{E}_\infty $-ring spectrum \cite[8.12 (i) and 8.13]{GGN}.
\end{pg}

\begin{pg}
The Picard category $\mls Pic ^{\mathbf{Z}}_{(\BaseSite, \mls O)}$ also has a multiplicative structure given by
\begin{equation}\label{E:5.1.1}
(r, \mls L)\otimes (r', \mls L'):= (rr', \mls L^{\otimes r'}\otimes \mls L^{\prime \otimes r}).
\end{equation}
Note here that $\mls L^{\otimes r'}$ and $\mls L^{\prime \otimes r}$ are defined by first defining them on the level of modules over rings and then globalizing.

This multiplicative structure can be upgraded to a structure of an $\mathbf{E}_\infty $-ring spectrum as follows.
\end{pg}

\begin{pg} 
Let $\BaseSite^\zar $ be as in \cref{P:3.8}, and consider again the  functors
$$
\coprod _nBGL_n:\BaseSite^{\zar , \op} \rightarrow (\text{$\mathbf{E}_\infty $-monoids}), \ \ (U, V)\mapsto \coprod _nBGL_n(\mls O^\zar (U, V)),
$$
and 
$$
\mathbf{Z}\times BGL_1:\BaseSite^{\zar , \op} \rightarrow (\text{$\mathbf{E}_\infty $-monoids}), \ \ (U, V)\mapsto \mathbf{Z}\times BGL_1(\mls O^\zar (U, V)).
$$

Let
$$
\mls K^\zar :\BaseSite^{\zar , \op} \rightarrow (\text{$\mathbf{E}_\infty $-monoids})
$$
be the sheafification of the group completion of $\coprod _nBGL_n$ and note that $\underline {\mls Pic}_{\mls O}^{\mathbf{Z}, \zar}$ is the sheafification of $\mathbf{Z}\times BGL_1$. Then the determinant defines a map of sheaves of $\mathbf{E}_\infty $-monoids
$$
\det :\mls K^\zar \rightarrow \underline {\mls Pic}_{\mls O}^{\mathbf{Z}, \zar}.
$$
\end{pg}

\begin{lem}\label{R:6.4} The determinant map induces an equivalence
$$
\tau _{\leq 1}\mls K^\zar \simeq \underline {\mls Pic}_{\mls O}^{\mathbf{Z}, \zar}.
$$
\end{lem}
\begin{proof}
It suffices to show that for a given $U\in \BaseSite $ the induced map of sheaves on $\Sp (\mls O(U))$ is an equivalence.  This reduces the proof of the lemma to the case of the Zariski topology of an affine scheme.  The verification in this case reduces immediately to the calculation of $K_0$ and $K_1$ for a local ring, and for such a ring $R$ we have $K_0(R) = \mathbf{Z}$ and $K_1(R) = R^*$.
\end{proof}

\begin{pg} Now observe that $\mls K^\zar $ has the structure of a sheaf of $\mathbf{E}_\infty $-rings, and therefore so does $\tau _{\leq 1}\mls K^\zar $.  In this way, $\underline {\mls Pic}_{\mls O}^{\mathbf{Z}, \zar }$, and therefore also $\underline {\mls Pic}_{\mls O}^{\mathbf{Z}}$ and $\mls Pic ^{\mathbf{Z}}_{(\BaseSite, \mls O)}$ are given $\mathbf{E}_\infty $-ring structures.  Note also that the underlying multiplication map is induced by the natural maps on $\coprod _nBGL_n$ and $\coprod _n BGL_1$, and therefore the underlying multiplicative structure on $\mls Pic^{\mathbf{Z}}_{(\BaseSite, \mls O)}$ is given by \eqref{E:5.1.1}.

Furthermore, if  $\mls K$ denote the sheaf on $\BaseSite$ associated to the presheaf $\gamma (\mls K^\zar )$.  We then get an equivalence
$$
\tau _{\leq 1}\mls K\simeq \underline {\mls Pic}_{(\BaseSite, \mls O)}^{\mathbf{Z}}.
$$
\end{pg}

This discussion implies the following:

\begin{thm}\label{T:multiplicative} The Picard category $\mls Pic _{(\BaseSite, \mls O)}^{\mathbf{Z}}$ has the structure of an $\mathbf{E}_\infty $-ring spectrum with multiplicative structure given by \eqref{E:5.1.1} and such that the map
\begin{equation}\label{E:5.3.1}
K(\BaseSite, \mls O)\rightarrow \mls Pic ^{\mathbf{Z}}_{(\BaseSite, \mls O)}
\end{equation}
is a map of $\mathbf{E}_\infty $-rings.
\end{thm}
\begin{proof}
This follows from the preceding discussion,
 and the observation that the map
$$
K(\BaseSite, \mls O)\rightarrow R\Gamma (\BaseSite , \mls K)
$$
is a map of $\mathbf{E}_\infty $-rings, by the universal property of group completion and the fact that the isomorphism
$$
\mls D_\perf (\BaseSite, \mls O)^\simeq  = R\Gamma (\BaseSite, \underline {\mls D}_\perf ^\simeq )
$$
is an isomorphism of $\mathbf{E}_\infty $-semirings.
\end{proof}


\section{The trace map}\label{S:trace}

\begin{pg}
Let $(\BaseSite, \mls O)$ be a ringed site and let $E$ be a perfect complex of $\mls O$-modules.  In this section we record some observations about the trace map, defined in \cite[V, 3.7.3]{Illusie},
$$
\mathrm{tr}_{I_\bullet }:\mls RHom (E, E\lotimes N(I_\bullet) )\rightarrow N(I_\bullet )
$$
for a simplicial $\mls O$-module $I_\bullet $.  This map is obtained as the composition of the inverse of the isomorphism (using the perfection of $E$)
$$
E^\vee \lotimes (E\lotimes N(I_\bullet ))\rightarrow \mls RHom (E, E\lotimes N(I_\bullet ))
$$
and the evaluation map
$$
E^\vee \lotimes (E\lotimes N(I_\bullet ))\rightarrow N(I_\bullet ).
$$
Observe that under the natural identification
$$
\mls RHom (E, E)\lotimes N(I_\bullet )\simeq \mls RHom (E, E\lotimes N(I_\bullet ))
$$
the map $\mathrm{tr}_{I_\bullet }$ is identified with the map $\mathrm{tr}_{\mls O}\otimes N(I_\bullet )$.  We often drop the subscript and write simply $\mathrm{tr}$ for $\mathrm{tr}_{I_\bullet }$ if no confusion seems likely to arise.
\end{pg}

\begin{pg}\label{P:7.2}
Fix a perfect complex $E$. Denote by $\mathrm{Mod}_{\mls O}$ (resp. $\mathrm{Mod}_{\mls O}^{\Delta ^\op }$) the category of sheaves of $\mls O$-modules (resp. sheaves of simplicial $\mls O$-modules).   We then have two functors
$$
F_1, F_2:\mathrm{Mod}_{\mls O}^{\Delta ^\op }\rightarrow (\text{$\mathrm{Sp}$-valued sheaves}) 
$$
given by
$$
F_1 (I_\bullet ) := \mathrm{DK}(\mls RHom (E, E\lotimes N(I_\bullet ))[1])
$$
and 
$$
F_2(I_\bullet ):= BI_\bullet .
$$
The trace map defines a morphism of $\infty$-functors
\begin{equation}\label{E:tracemap}
\mathrm{tr}:F_1\rightarrow F_2.
\end{equation}
Now observe that $F_1$ is the left Kan extension of its restriction to $\mathrm{Mod}_{\mls O}$ (this follows from the observation that the normalization of a simplicial abelian group is quasi-isomorphic to the homotopy colimit), and therefore $\mathrm{tr}$ is determined by the restrictions of these functors to $\mathrm{Mod}_{\mls O}$.

In fact, from the perfection of $E$ we get slightly more.  Namely, note that since $E$ is perfect we have locally
$$
\mls RHom (E, E\lotimes N(I_\bullet )[n+1])\simeq \tau _{\leq 0}\mls RHom(E, E\lotimes N(I_\bullet )[n+1])
$$
for $n$ sufficiently large.   Therefore, if we write $F_1^{\leq 0}$ for the functor
$$
I_\bullet \mapsto \mathrm{DK}(\tau _{\leq 0}\mls RHom (E, E\lotimes N(I_\bullet ))[1])
$$
then $F_1$ is isomorphic to the functor
$$
I_\bullet \mapsto \text{colim}_n\Omega ^nF_1^{\leq 0}(B^nI_\bullet ),
$$
where $B^nI_\bullet $ is the $n$-fold delooping of $I_\bullet $, corresponding under the Dold-Kan correspondence to $N(I_\bullet )[n]$. It follows that a morphism of $\infty $-functors $F_1^{\leq 0}\rightarrow F_2$ can be extended to a morphism $F_1\rightarrow F_2$, and therefore is determined by its restriction to $\text{Mod}_{\mls O}$.
\end{pg}

\begin{pg}
To understand the map $\mathrm{tr}$ on modules, consider first the case of  the punctual topos, a ring $A$, and an $A$-module $I$.   Let $\mathrm{Perf }^\strict _{A[I]}$ denote the category of strictly perfect $A[I]$-modules, viewed as a Waldhausen category as in \cite[3.1]{TT}, and let $\mathrm{EXT}(A, I)$ be the Picard category of short exact sequences of $A$-modules
$$
0\rightarrow I\rightarrow T\rightarrow A\rightarrow 0,
$$
as in \cite[Expos\'e XVIII, 1.4.22]{SGA4}.  Note that by \cite[Expos\'e XVIII 1.4.23]{SGA4} the object of $\mathrm{Sp}^{\geq 0}$ assocated to $\mathrm{EXT}(A, I)$ is $BI$.

There is a determinant map (in the sense of \cref{P:3.5})
$$
\delta :\mathrm{Perf}^\strict _{A[I]}\rightarrow \mathrm{EXT}(A, I).
$$

For an object $E'\in \mathrm{Perf}_{A[I]}^\strict $ with reduction $E\in \mathrm{Perf}^\strict _A$ we get an exact sequence of complexes of $A$-modules
$$
0\rightarrow E\otimes I\rightarrow E'\rightarrow E\rightarrow 0.
$$
Tensoring with $E^\vee $, pulling back along $\mathrm{id}:A\rightarrow E\otimes E^\vee $ and pushing out along the trace map we get an object of $\mathrm{EXT}(A, I)$; in a diagram:
$$
\xymatrix{
0\ar[r]& I\ar[r]& T\ar[r]& A\ar[r]& 0\\
0\ar[r]&E\otimes E^\vee \otimes I\ar@{=}[d]\ar[u]_{\mathrm{tr}}\ar[r]& \mls E'\ar[u]\ar[r]\ar[d]& A\ar@{=}[u]\ar[r]\ar[d]^-{\mathrm{id}}& 0\\
0\ar[r]& E\otimes E^\vee \ar[r]& E'\otimes E^\vee \ar[r]& E\otimes E^\vee \ar[r]& 0.}
$$

To extend this construction to a determinant functor,  note that by \cite[2.3]{Knudsen}  it suffices to define a determinant functor on the category of projective $A[I]$-modules $\mathrm{Proj}_{A[I]}$ with appropriate properties.    

The preceding construction defines a functor
$$
\mathrm{iso}(\mathrm{Proj}_{A[I]})\rightarrow \mathrm{EXT}(A, I).
$$

Next consider a short exact sequence of projective $A[I]$-modules
$$
0\rightarrow E_1'\rightarrow E'\rightarrow E_2'\rightarrow 0
$$
with reduction 
$$
0\rightarrow E_1\rightarrow E\rightarrow E_2\rightarrow 0.
$$
Set
$$
\Sigma := \mathrm{Ker}(E\otimes E^\vee \rightarrow E_2\otimes E_1^\vee ).
$$
Then the two maps
$$
E\otimes E^\vee \rightarrow E_2\otimes E^\vee , \ \ E\otimes E^\vee \rightarrow E\otimes E_1^\vee 
$$
induce a map
$$
\rho :\Sigma \rightarrow E_2\otimes E_2^\vee \oplus E_1\oplus E_1^\vee .
$$
Furthermore, if 
$$
0\rightarrow E\otimes E^\vee \otimes I\rightarrow \mls E\rightarrow A\rightarrow 0
$$
is the extension obtained from $E'$ then the pushout of this extension along the map
$$
E\otimes E^\vee \rightarrow E_2\otimes E_1^\vee 
$$
is canonically trivialized, which implies that $\mls E$ is obtained from an extension
$$
0\rightarrow \Sigma \otimes I\rightarrow \mls E_\Sigma \rightarrow A\rightarrow 0.
$$
Furthermore, this identifies the pushout of $\mls E_\Sigma $ along $\rho $ with the sum of the extensions obtained from $E_1'$ and $E_2'$.  In this way we obtain a predeterminant functor in the sense of \cite[1.2]{Knudsen}.   We leave it to the reader to verify that this in fact defines a determinant.
\end{pg}

\begin{pg}
Combining this with the discussion in \cref{P:3.5} we obtain a map
$$
\bar t:K(A[I])\rightarrow BI
$$
from which one can recover  the trace map
$$
\mathrm{tr}_{I[1]}:\mathrm{DK}(\tau _{\leq 0}\mls RHom (E, E\lotimes I)[1])\rightarrow BI
$$
as the composition 
$$
\xymatrix{
\mathrm{DK}(\tau _{\leq 0}\mls RHom (E, E\lotimes I)[1])\ar[r]^-{\cref{P:2.15}}& \mls D_\perf ^{\strict }(A[I])^\simeq \ar[r]& K(A[I])\ar[r]^-{\bar t}& BI.}
$$

Combining \cref{P:7.2} and \cref{P:3.7} this also defines, by passing to left Kan extensions, a map for every simplicial $A$-module $I_\bullet $ 
$$
\bar t:K(A[I_\bullet ])\rightarrow BI_\bullet 
$$
inducing the trace map $\mathrm{tr}_{I_\bullet [1]}$.
\end{pg}

\begin{pg}\label{P:7.5b}
In the case of a general ringed topos $(\BaseSite, \mls O)$ and an $\mls O$-module $I$ we get by functoriality of the preceding constructions a morphism of presheaves taking values in $\mathrm{Sp}^{\geq 0}$
$$
\bar t:\mls K_{\mls O[I]}\rightarrow BI,
$$
where $\mls K_{\mls O[I]}$ is the presheaf sending $U\in \BaseSite$ to the $K$-theory of strictly perfect complexes of $\mls O(U)[I(U)]$-modules, such that the composition
$$
\xymatrix{
\mathrm{DK}(\tau _{\leq 0}\mls RHom (E, E\lotimes I[1]))\ar[r]& \underline {\mls D}_{\perf, \mls O[I]}^{\strict , \simeq }\ar[r]&\mls K_{\mls O[I]}\ar[r]^-{\bar t}& BI}
$$
is the trace map.   Using the method of \cref{P:7.2} we then also get a map for a simplicial $\mls O$-module $I_\bullet $
$$
\bar t:\mls K_{\mls O[I_\bullet ]}\rightarrow BI_\bullet 
$$
inducing the trace map on $\tau _{\leq 0}\mls RHom (E, E\lotimes N(I_\bullet )[1]).$
\end{pg}

\section{Determinants and traces}\label{S:section6}

In this section we elucidate the relationship between the determinant map and the trace map constructed in \cite[V, 3.7.3]{Illusie}.  

\begin{pg} We begin the discussion in the punctual case.  Let $A$ be a ring and let $I_\bullet $ be a simplicial $A$-module with associated simplicial ring of dual numbers $A[I_\bullet ]$.

Note that $\pi _0(A[I_\bullet ])\simeq A[\pi _0(I_\bullet )]$ and we have a short exact sequence of simplicial monoids
$$
1\rightarrow 1+I_\bullet \rightarrow GL_1(A[I_\bullet ])\rightarrow A^*\rightarrow 1.
$$
In particular, $GL_1(A[I_\bullet ])$ is a simplicial group.  Furthermore, the retraction $r$ defines a splitting of this sequence giving a homomorphism
\begin{equation}\label{E:3.8.1}
GL_1(A[I_\bullet ])\rightarrow I_\bullet .
\end{equation}
Concretely this is given by writing an element $\alpha \in (A[I_n])^*$ as $\bar \alpha (1+x)$, where $\bar \alpha \in A^*$ is the image of $\alpha $ in $A^*$, and then sending $\alpha $ to $x$.
\end{pg}

\begin{pg}
Note that because $-1\in A^*\subset GL_1(A[I_\bullet ])$, the projection map \eqref{E:3.8.1} induces a map
$$
\mls Pic^{\mathbf{Z}}(A[I_\bullet ])\rightarrow BI_\bullet ,
$$
compatible with the symmetric monoidal structure.  Furthermore, the induced map
\begin{equation}\label{E:4.3.1}
\mls Pic ^{\mathbf{Z}}(A[I_\bullet ])\rightarrow \mls Pic ^{\mathbf{Z}}(A)\times BI_\bullet 
\end{equation}
is an equivalence.  The determinant map
$$
\mdet _{A[I_\bullet ]}:\Perf ({A[I_\bullet ]})^\simeq \rightarrow \mls Pic ^{\mathbf{Z}}(A[I_\bullet ])
$$
can therefore be written as 
$$
(\mdet _A, t):\Perf ({A[I_\bullet ]})^\simeq \rightarrow \mls Pic ^{\mathbf{Z}}(A)\times BI_\bullet ,
$$
where the first component is given by the projection
$$
\Perf ({A[I_\bullet ]})^\simeq \rightarrow \Perf (A)^\simeq 
$$
followed by the determinant map for $A$-modules, and $t$ is a symmetric monoidal map
$$
t:\Perf ({A[I_\bullet ]})^\simeq \rightarrow BI_\bullet .
$$
\end{pg}

\begin{pg}
Now consider a ringed site $(\BaseSite, \mls O)$, and a simplicial $\mls O$-module $I_\bullet $.  Write $R\Gamma ^\Delta (I )$ for the simplicial object of the derived $\infty $-category  obtained by taking derived functors of the global section functor.  In terms of the normalization functor from simplicial modules to complexes we have
$$
N(R\Gamma ^\Delta (I_\bullet ))\simeq \tau _{\leq 0}R\Gamma (\BaseSite, N(I_\bullet )),
$$
where the right side denotes the usual derived functor cohomology \cite[I, 3.2.1.11]{Illusie}.

  Proceeding object by object and taking limits we get an equivalence 
$$
\mls Pic _{(\BaseSite, \mls O[I_\bullet ])}^{\mathbf{Z}}\simeq \mls Pic _{(\BaseSite, \mls O)}^{\mathbf{Z}}\times R\Gamma ^\Delta B (I_\bullet) 
$$ 
which gives a description of the determinant map 
$$
\mdet _{(\BaseSite, \mls O[I_\bullet ])} = (\mdet _{(\BaseSite, \mls O)}, t):\Perf {(\BaseSite, \mls O[I_\bullet ])}^\simeq \rightarrow \mls Pic _{(\BaseSite, \mls O)}^{\mathbf{Z}}\times R\Gamma ^\Delta B (I_\bullet ),
$$
where 
$$
t:\Perf {(\BaseSite, \mls O[I_\bullet ])}^\simeq \rightarrow R\Gamma ^\Delta B (I_\bullet )
$$
is a map of symmetric monoidal $\mathbf{E}_\infty $-categories.
\end{pg}

\begin{pg}
Let 
$$
q:\Perf {(\BaseSite, \mls O[I_\bullet ])}\rightarrow \Perf {(\BaseSite, \mls O)}
$$
be the projection.  By \cref{P:2.15} the fiber of $q$ over the point given by an object $E\in \Perf {(\BaseSite, \mls O)}$ is given by
$$
\mathrm{DK}( \tau _{\leq 0}\mathrm{RHom} (E, E\lotimes N(I_\bullet )[1]).
$$
 Restriction $t$ to the fiber of $q$ and using the Dold-Kan correspondence we get a map
\begin{equation}\label{E:4.5.1}
\tau _{\leq 0}\mathrm{RHom} (E, E\lotimes N(I_\bullet ))\rightarrow \tau _{\leq 0}R\Gamma (N(I_\bullet )),
\end{equation}
in the derived category.
\end{pg}

\begin{rem} For $n\geq 0$ write $I_\bullet [n]$ for $\mathrm{DK}(N(I_\bullet )[n])$.  Then 
$$
\tau _{\leq 0}\mathrm{RHom} (E, E\lotimes N(I_\bullet [n]))\simeq (\tau _{\leq n}\mathrm{RHom}(E, E\lotimes N(I_\bullet )))[n],
$$
and therefore by shifting we get a map
$$
\tau _{\leq n} \mathrm{RHom} (E, E\lotimes N(I_\bullet ))\rightarrow R\Gamma (N(I_\bullet ))
$$
for all $n$.  One can show directly that these maps are compatible and therefore by taking colimits define a map
$$
\mathrm{ RHom }(E, E\lotimes N(I_\bullet ))\rightarrow R\Gamma (N(I_\bullet )).
$$
This compatibility follows, however, from \cref{P:4.6.1} below so we do not elaborate further on this point here.
\end{rem}

\begin{prop}\label{P:4.6.1} The map \eqref{E:4.5.1} agrees with the trace map defined in \cite[V, 3.7.3]{Illusie}.
\end{prop}
\begin{proof}
  The basic idea is to construct a second map
$$
t':K(\BaseSite, {\mls O[I_\bullet ]}) \rightarrow R\Gamma ^\Delta B (I_\bullet) 
$$
 which induces the trace map on the fibers of $q$, and then show that $t= t'$ using the universal property of group completion.

For this  we extend the preceding constructions to presheaves on $\BaseSite^\zar $, defined as in \cref{P:3.8}.

First of all, repeating the construction giving \eqref{E:4.3.1} we get an equivalence
$$
\underline {\mls Pic}^{\mathbf{Z}, \zar }_{\mls O[I _\bullet ]}\rightarrow \mls Pic^{\mathbf{Z}, \zar }_{\mls O}\times BI_\bullet ^\zar ,
$$
and the determinant map on $\underline {\mls D}^{\prime , \zar , \simeq  }_{\mathrm{perf}, \mls O[I_\bullet ]}$ breaks into two parts
$$
(\mdet _{\mls O}^\zar , t^\zar ):\underline {\mls D} _{\mathrm{perf}, \mls O[I_\bullet ]}^{\prime , \zar , \simeq }\rightarrow \mls Pic ^{\mathbf{Z}, \zar }_{\mls O}\times BI_\bullet ^\zar .
$$
Similarly, running through the construction of \cref{P:7.5b} we get a map
$$
t^{\prime , \zar }:\underline {\mls D} _{\mathrm{perf}, \mls O[I_\bullet ]}^{\prime , \zar , \simeq }\rightarrow BI_\bullet ^\zar 
$$
inducing the trace map after applying $\gamma $, sheafifying, and taking global sections.

It therefore suffices to show that $t^\zar $ and $t^{\prime \zar }$ agree.  For this it suffices, in turn, to show that the restrictions along
$$
\coprod _nBGL_n(\mls O^\zar [I^\zar _\bullet ])\rightarrow BI_\bullet ^\zar 
$$
agree. This is immediate from the constructions, and we get \cref{P:4.6.1}. 
\end{proof}

\section{Deformations of complexes}\label{S:section7}

\begin{pg}\label{P:9.1} Let $\BaseSite$ be a site and let $\mls O'\rightarrow \mls O$ be a  surjection of sheaves of rings on $\BaseSite$ with kernel $K$, a square zero ideal.

Let $E\in D(\BaseSite, \mls O)$ be a perfect complex of $\mls O$-modules.  As in \cref{P:2.12} we can then consider the category $\Def (E)$ of deformations of $E$ to $\mls O'$. 

The following is well-known in many cases (e.g. \cite[IV, 3.1.5]{Illusie}, \cite[3.1.1]{deformingkernels}).
\end{pg}

\begin{thm}\label{T:5.2} Let $E$ be a perfect complex of $\mls O$-modules on $\BaseSite$.

(i) There is a class $\omega (E)\in \mathrm{Ext}^2(E, E\lotimes K)$ which vanishes if and only if $E$ lifts to a perfect complex of $\mls O'$-modules.

(ii) If $\omega (E) = 0$ then the set of isomorphism classes of liftings form a torsor under $\mathrm{Ext}^1(E, E\lotimes K).$

(iii) If $\mathrm{Ext}^{-1}(E, E) = 0$  then the set of automorphisms of any lifting is canonically identified with $\mathrm{Ext}^0(E, E\lotimes K)$.
\end{thm}

\begin{rem} 
If $E'$ is a deformation of $E$ to $\mls O'$, then by applying $R\mathrm{Hom}_{\mls O'}(E', -)$
to the distinguished triangle
$$
K\lotimes _{\mls O}E\rightarrow E'\rightarrow E\rightarrow K\lotimes _{\mls O}E[1]
$$
we get a boundary map
$$
\partial _{E'}:\mathrm{Ext}^{-1}_{\mls O}(E, E)\rightarrow \mathrm{Ext}^0(E, E\lotimes _{\mls O}K).
$$  
If we don't assume that $\mathrm{Ext}^{-1}(E, E) = 0$ then the group of automorphisms of $(E', \sigma )$ is canonically isomorphic to the cokernel of $\partial _{E'}$.  Note that this group may depend on $E'$.
\end{rem} 

The proof of \cref{T:5.2} occupies the remainder of this section.

\begin{pg}
Statements (ii) and (iii) follow from the discussion in \cref{P:2.12}.
Indeed, if $\Def _\infty (E)$ denotes the $\infty $-categorical fiber product of the diagram
$$
\xymatrix{
& \Perf {(\BaseSite, \mls O')}\ar[d]\\
\star \ar[r]^-E& \Perf {(\BaseSite, \mls O)}}
$$
then we constructed in \cref{P:2.12} a functor
$$
[\Def _\infty (E)]\rightarrow \Def (E),
$$
where $[\Def _\infty (E)]$ denotes the underlying $1$-category of $\Def _\infty (E)$, which induces a bijection on isomorphism classes of objects, and if $\mathrm{Ext}^{-1}(E, E) = 0$ is an equivalence of categories.  On the other hand, if there exists a lifting of $E$ then by \cref{P:2.15} we have
\begin{equation}\label{E:6.4.1}
\Def _\infty (E)\simeq \mathrm{DK}(\tau _{\leq 0}\mathrm{RHom}(E, E\lotimes K)[1]).
\end{equation}
This identification depends on the choice of a lifting of $E$ to $\mls O'$, but immediately implies (iii).
\end{pg}

\begin{pg}\label{P:6.5}
To understand the dependence of \eqref{E:6.4.1} on the choice of a lifting we can use classical techniques to get an action (suitably defined) of the right side of \eqref{E:6.4.1} on $\Def _\infty (E)$.

For two surjections $\mls O'_i\rightarrow \mls O$  ($i=1,2$) with square zero kernels $K_i$, the natural functor
\begin{equation}\label{E:6.5.0}
\Def _{\infty , \mls O'\times _{\mls O}\mls O^{\prime \prime }}(E)\rightarrow \Def _{\infty , \mls O'}(E)\times \Def _{\infty ,  \mls O^{\prime \prime }}(E)
\end{equation}
is an equivalence if both sides are nonempty, where in the subscripts we indicate which square-zero surjection to $\mls O$ we are considering.  Now observe that
$$
\mls O'\times _{\mls O}\mls O'\simeq \mls O'[K]\simeq \mls O'\times _{\mls O}\mls O[K].
$$
We therefore get a map
\begin{equation}\label{E:5.5.1}
\Def _{\infty , \mls O'}(E)\times  \Def _{\infty , \mls O[K]}(E)\rightarrow \Def _{\infty , \mls O'}(E).
\end{equation}
The retraction $\mls O\rightarrow \mls O[K]$ induces a canonical lifting of $E$ and therefore using \eqref{E:6.4.1} we have a canonical isomorphism
$$
\Def _{\infty , \mls O[K]}(E)\simeq \mathrm{DK}(\tau _{\leq 0}\mathrm{RHom}(E, E\lotimes K)[1]),
$$
and \eqref{E:5.5.1} can be written as a map
$$
\Def _{\infty , \mls O'}(E)\times \mathrm{DK}(\tau _{\leq 0}\mathrm{RHom}(E, E\lotimes K)[1])\rightarrow \Def _{\infty , \mls O'}(E).
$$
Passing to isomorphism classes we get an action of $\mathrm{Ext}^1(E, E\lotimes K)$ on the set of isomorphism classes in $[\Def _\infty (E)]$, and this action is simply transitive when there exists a lifting in light of the isomorphism \eqref{E:6.4.1}.  This gives (ii).
\end{pg}

\begin{pg}\label{P:5.6}
To define the obstruction we use simplicial techniques.    As noted in the introduction, the work in this article is naturally viewed in the context of formal moduli problems in the sense of \cite{LurieSAG}. The interested reader may wish to consult the introduction to Chapter IV in \cite{LurieSAG} for more on this perspective.

Choose an inclusion 
$$
K\hookrightarrow J
$$
with $J$ an injective $\mls O$-module and let $I_\bullet $ denote the simplicial $\mls O$-module corresponding to the two-term complex 
\begin{equation}\label{E:7.6.1}
\xymatrix{
J\ar[r]^-{\mathrm{id}_J}& J}
\end{equation}
concentrated in degrees $-1$ and $0$.  So we have an inclusion of simplicial $\mls O$-modules $K\hookrightarrow I_\bullet $.  Let $\overline I_\bullet $ denote the cokernel.  The simplicial module $\overline I_\bullet $ is the simplicial module associated to the two term complex
$$
J\rightarrow J/K,
$$
which is quasi-isomorphic to $K[1]$.

Let $\widetilde {\mls O}$ denote the simplicial ring obtained by pushout from the diagram
$$
\xymatrix{
K\ar@{^{(}->}[r]\ar@{^{(}->}[d]& \mls O'\\
I_\bullet .&}
$$
So $\widetilde {\mls O}$ comes equipped with a surjective map to $\mls O$ with kernel $I_\bullet $.
Note also that the further pushout of $\widetilde {\mls O}$ along $I_\bullet \rightarrow \overline {I}_\bullet $ is canonically isomorphic to $\mls O[\overline I_\bullet ]$.
\end{pg}

\begin{rem} Note that here we are using the Dold-Kan correspondence applied to the complex \eqref{E:7.6.1} and then forming the pushout in the category of simplicial rings to obtain $\widetilde {\mls O}$. In characteristic $0$ one could also first consider the pushout in the category of commutative differential graded algebras, but in general it is preferable to work in the category of commutative simplicial rings.
\end{rem}

\begin{lem}\label{L:5.7}
The natural map $\Perf (\BaseSite, {\widetilde {\mls O}})\rightarrow \Perf (\BaseSite, {\mls O})$ is an equivalence.
\end{lem}
\begin{proof}
This follows from \cite[4.4]{SS}, which implies that the functor induces an equivalence of homotopy categories, and the description of the fibers given in \cref{P:2.15}.
\end{proof}

\begin{pg}
Given a perfect complex $E$ of $\mls O$-modules, we therefore get a perfect complex $\widetilde E$ of $\widetilde {\mls O}$-modules.  Pushing out this complex $\widetilde E$ along the natural map
$$
\widetilde {\mls O}\rightarrow \mls O[\overline I_\bullet ]
$$
we get a class in 
$$
\mathrm{Ext}^1_{\mls O}(E, E\lotimes \overline I_\bullet )\simeq \mathrm{Ext}^1_{\mls O}(E, E\lotimes K[1])\simeq \mathrm{Ext}^2_{\mls O}(E, E\lotimes K).
$$
We define  
$$
\omega (E)\in \mathrm{Ext}^2_{\mls O}(E, E\lotimes K)
$$
to be this class.  We will see in \cref{SS:9.15} below that the class $\omega (E)$ vanishes if and only if $E$ lifts to $\mls O'$.
\end{pg}

\subsection{Gabber's construction of the obstruction}\label{S:section9}

Gabber has given a construction of an obstruction to deforming a perfect complex which is more direct \cite{Gabber}.  The equivalence with the definition of the class $\omega (E)$ can be see as follows.

 Let $\mls O'\rightarrow \mls O$ and $E$ be as in \cref{P:9.1}.
Gabber defines a class
$$
o(E)\in \mathrm{Ext}^2_{\mls O}(E, E\otimes _{\mls O}^{\mathbf{L}}K),
$$
which vanishes if and only if there exists a deformation of $E$ to $\mls O'$. In fact, Gabber's construction is more general and can also be considered for non-perfect complexes.

\begin{pg}\label{P:7.5} The class $o(E)$ is constructed as follows.  Assume that $E$ is represented by a bounded above complex of flat modules.  We will work directly with complexes.  

Choose a bounded above complex $G$ of flat $\mls O'$-modules which is acyclic and a surjective map
$$
G\rightarrow E.
$$

Let $S$ be the kernel so we have
$$
0\rightarrow S\rightarrow G\rightarrow E\rightarrow 0.
$$
From the snake lemma applied to the diagram
$$
\xymatrix{
& S\otimes K\ar[d]\ar[r]& G\otimes K\ar[r]\ar[d]& E\otimes K\ar[d]^0\ar[r]& 0\\
0\ar[r]& S\ar[r]& G\ar[r]& E\ar[r]& 0}
$$
we get an exact sequence of complexes
\begin{equation}\label{E:7.5.1}
0\rightarrow E\otimes _{\mls O}K\rightarrow S\otimes _{\mls O'}\mls O\rightarrow G\otimes _{\mls O'}\mls O\rightarrow E\rightarrow 0.
\end{equation}
Let $o(E)\in \text{Ext}^2_{\mls O}(E, E\otimes _{\mls O}K)$ denote the class of this Yoneda extension.  A straightforward exercise shows that this class is independent of the choice of $G\rightarrow E$.
\end{pg}

\begin{prop} $o(E) = \omega (E)$.
\end{prop}
\begin{proof}
We proceed with notation as in \cref{P:5.6}.  Let $\widetilde E\rightarrow E$ be the complex of $\widetilde {\mls O}$-modules over $E$ provided by the equivalence in \cref{L:5.7}.   Abusing notation, we consider this as a complex of $N(\widetilde {\mls  O})$-modules. Choose $G\rightarrow \widetilde E$ a surjection of complexes of $\mls O'$-modules with $G$ an acyclic bounded above complex of flat $\mls O'$-modules.  We then get a commutative diagram
$$
\xymatrix{
0\ar[r]& S\ar[d]\ar[r]& G\ar[d]\ar[r]& E\ar@{=}[d]\ar[r]& 0\\
0\ar[r]& E\otimes _{\mls O}N(I_\bullet )\ar[r]& \widetilde E\ar[r]& E\ar[r] & 0.}
$$

Tensoring with $\mls O$ as above we get a commutative diagram
$$
\xymatrix{
0\ar[r]& E\otimes _{\mls O}K\ar@{^{(}->}[rd]\ar[r]& S\otimes _{\mls O'}\mls O\ar[d]\ar[r]& G\otimes _{\mls O'}\mls O\ar[dd]\ar[r]& E\ar@{=}[dd]\ar[r]& 0\\
&& E\otimes N(I_\bullet )\ar[d]&&&\\
&0\ar[r]& E\otimes N(\overline I_\bullet)\ar[r]& \widetilde E\otimes _{N(\widetilde {\mls  O})}N(\mls O[\overline I_\bullet ]))\ar[r]& E\ar[r]& 0.}
$$
From this it follows that $o(E)$ is the image under the boundary map
$$
\mathrm{Ext}^1_{\mls O}(E, E\lotimes N(\overline I_\bullet ))\rightarrow \mathrm{Ext}^2_{\mls O}(E, E\lotimes K)
$$
of the class in $\mathrm{Ext}^1_{\mls O}(E, E\lotimes N(\overline I_\bullet ))$ given by the extension
$$
0\rightarrow E\otimes _{\mls O}N(\overline I_\bullet )\rightarrow \widetilde E\otimes _{N(\widetilde {\mls O})}N(\mls O[\overline I_\bullet ])\rightarrow E\rightarrow 0,
$$
which in turn implies that $o(E) = \omega (E)$.
\end{proof}

\subsection{$\omega (E) = 0$ if and only if $E$ lifts}\label{SS:9.15}

\begin{prop} The class $\omega (E)$ is $0$ if and only if $E$ lifts to $\mls O'$.
\end{prop}
\begin{proof}
We show this using Gabber's description of $\omega (E)$.

If $E$ lifts to a complex $E'$ over $\mls O'$ then as above we can choose $G\rightarrow E$ that factors through a map $G\rightarrow E'$.    Examining the commutative diagram
$$
\xymatrix{
0\ar[r]& S\ar[d]\ar[r]& G\ar[d]\ar[r]& E\ar[r]\ar@{=}[d]& 0\\
0\ar[r]& E\otimes _{\mls O}K\ar[r]& E'\ar[r]& E\ar[r]& 0.}
$$
The map $S\rightarrow E\otimes _{\mls O}K$ induces a retraction of the inclusion $E\otimes _{\mls O}K\hookrightarrow S\otimes _{\mls O'}\mls O$ and therefore the corresponding Yoneda class $\omega (E)$ is zero.

Conversely, suppose $\omega (E) = 0$.  Fix $G\rightarrow E$ as above, and let $T$ denote the image of the map
$$
S\otimes _{\mls O'}\mls O\rightarrow G\otimes _{\mls O'}\mls O.
$$
We then have short exact sequences
$$
0\rightarrow E\otimes _{\mls O}K\rightarrow S\otimes _{\mls O'}\mls O\rightarrow T\rightarrow 0
$$
and
$$
0\rightarrow T\rightarrow G\otimes _{\mls O'}\mls O\rightarrow E\rightarrow 0.
$$
The first of these defines a class
$$
\alpha \in \text{Ext}^1_{\mls O}(T, E\otimes _{\mls O}K)
$$
whose image under the boundary map
$$
\text{Ext}^1_{\mls O}(T, E\otimes _{\mls O}K)\rightarrow \text{Ext}^2_{\mls O}(E, E\otimes _{\mls O}K)
$$
defined by the second sequence is the class $\omega (E)$.  Since $G$ is acyclic this boundary map is injective (since $\text{Ext}^1_{\mls O}(G\otimes _{\mls O'}\mls O, E\otimes _{\mls O}K)$ surjects onto the kernel), so the class $\alpha $ is $0$.  Thus there exists a morphism $r:S\otimes _{\mls O'}\mls O\rightarrow E\otimes _{\mls O}K$ in the derived category splitting the inclusion $E\otimes _{\mls O}K\hookrightarrow S\otimes _{\mls O'}\mls O$.  Forming the pushout of the diagram
$$
\xymatrix{
0\ar[r]& S\ar[d]^-r\ar[r]& G\ar[r]& E\ar[r]& 0\\
& E\otimes _{\mls O}K&&}
$$
we obtain $E'$ lifting $E$.
\end{proof}

\begin{pg}
In fact a bit more is true.
Note that the natural map of simplicial rings
$$
\mls O'\rightarrow \mls O\times _{\mls O[\overline I_\bullet ]}\widetilde {\mls O}
$$
is an isomorphism, so there is a commutative square
\begin{equation}\label{E:square}
\xymatrix{
\Perf (\BaseSite, {\mls O'})\ar[d]\ar[r]&\Perf (\BaseSite, {\widetilde {\mls O}})\ar[d]\\
\Perf (\BaseSite, {\mls O})\ar[r]& \Perf (\BaseSite, {\mls O[\overline I_\bullet ]})}
\end{equation}
\end{pg}

\begin{lem} The  square \eqref{E:square}
is homotopy cartesian.
\end{lem}
\begin{proof}
This follows from the preceding discussion  combined with \cref{P:2.15}, which implies that the map on fibers is an equivalence.  
\end{proof}

\begin{rem} In the case of the punctual topos, the above results are special cases of the  results in \cite[16.2]{LurieSAG}.
\end{rem}

\begin{rem}
We have formulated \cref{T:5.2} in classical terms using the category $\Def (E)$.  This forces, in particular, the statement of \cref{T:5.2} (iii) to incorporate the assumption of vanishing $\mathrm{Ext}^{-1}$.  The proof, however, shows, that if one considers instead the category $[\Def _\infty (E)]$, then statements (i)-(iii) all hold and no assumption of vanishing $\mathrm{Ext}^{-1}$ is needed in (iii).  Furthermore, if the obstruction is zero then the choice of a lifting of $E$ identifies $[\Def _\infty (E)]$ with the Picard category associated to the two-term complex
$$
\tau _{\geq -1}\tau _{\leq 0}(\mathrm{RHom}(E, E\lotimes K)[1]).
$$
\end{rem}

\section{Proof of \cref{T:maintheorem}}\label{S:section8}

We continue with the notation of the preceding section.

\begin{pg}
For statement (i) of \cref{T:maintheorem} note that there is a commutative diagram
\begin{equation}\label{E:7.1.1}
\xymatrix{
\Perf (\BaseSite, {\mls O})^\simeq \ar[d]^-{\mdet }& \Perf (\BaseSite, {\widetilde {\mls O}})^\simeq \ar[l]_-{\simeq }\ar[d]^-\mdet \ar[r]& \Perf (\BaseSite, {\mls O[\overline I_\bullet ]})^\simeq \ar[d]^-{\mdet }& \mathrm{DK}\tau _{\leq 0}\mathrm{RHom}(E, E\lotimes K[2])\ar[d]^-{\mathrm{tr}}\ar[l]\\
\mls Pic ^{\mathbf{Z}}_{\mls O}& \mls Pic^{\mathbf{Z}}_{\widetilde {\mls O}}\ar[l]_-{\simeq }\ar[r]& \mls Pic^{\mathbf{Z}}_{\mls O[\overline I_\bullet ]}& R\Gamma ^\Delta (BK[1])\ar[l].}
\end{equation}
By definition the obstruction class $\omega (E)$ is obtained as follows: Let $\widetilde E\in [\mls D_\perf (\BaseSite, \widetilde {\mls O})]$ be the corresponding object.  Then the pushout of $\widetilde E$ to $\mls O[\overline I_\bullet ]$ is a deformation of $E$ to $\mls O[\overline I_\bullet ]$ and therefore defines an isomorphism class in the fiber
$$
\mathrm{DK}(\tau _{\leq 0}\mathrm{RHom}(E, E\lotimes K[2])).
$$
The class $\omega (E)\in \mathrm{Ext}^2(E, E\lotimes K)$ is the class of this isomorphism class.

The obstruction class $\omega (\det (E))$ is obtained similarly from the bottom row of \eqref{E:7.1.1}.  Statement (i) in \cref{T:maintheorem} therefore follows from the commutativity of \eqref{E:7.1.1}.
\end{pg}

\begin{rem}\label{R:7.1.half}
Note that the construction of the obstruction class $\omega (\mls L)$ for an invertible $\mls O$-modules is additive in $\mls L$ in the sense that for two invertible $\mls O$-modules $\mls L$ and $\mls M$ we have
$$
\omega (\mls L\otimes \mls M) = \omega (\mls L)+ \omega (\mls M).
$$
\end{rem}

\begin{pg} Next we turn to statements (ii) and (iii) in \cref{T:maintheorem}.  Write $\mls Pic ^{\mathbf{Z}}_{\det (E), \mls O'}$ for the fiber product of the diagram
$$
\xymatrix{
& \mls Pic ^{\mathbf{Z}}_{\mls O'}\ar[d]\\
\star \ar[r]^-{\det (E)}& \mls Pic ^{\mathbf{Z}}_{\mls O},}
$$
and similarly for $\mls O[K]$ and $\mls O'\times _{\mls O}\mls O'$.  By \cref{P:4.6.1} the following diagram commutes
$$
\xymatrix{
\mls F_{E, \mls O'}\times \mls F_{E, \mls O'}\ar[d]^-{\mathrm{det} \times \mathrm{det}}& \mls F_{E, \mls O'\times _{\mls O}\mls O'}\ar[l]_-{\eqref{E:6.5.0}}\ar[d]^-{\mathrm{det} }\ar[r]^-{\simeq }& \mls F_{E, \mls O'}\times \mls F_{E, \mls O[K]}\ar[d]^-{\mathrm{det} \times \mathrm{det} }\\
\mls Pic ^{\mathbf{Z}}_{\mathrm{det} (E), \mls O'}\times \mls Pic ^{\mathbf{Z}}_{\det (E), \mls O'}& \mls Pic ^{\mathbf{Z}}_{\det (E), \mls O'\times _{\mls O}\mls O'}\ar[l]\ar[r]^-{\simeq }& \mls Pic ^{\mathbf{Z}}_{\det (E), \mls O'}\times \mls Pic ^{\mathbf{Z}}_{\det (E), \mls O[K]},}
$$
$$
\xymatrix{
\mls F_{E, \mls O'}\times \mls F_{E, \mls O[K]}\ar[d]^-{\mathrm{det} \times \mathrm{det} }\ar[r]^-{\simeq }& \mls F_{E, \mls O'}\times \mathrm{DK}\tau _{\leq 0}\mathrm{RHom}(E, E\lotimes K[1])\ar[d]^-{\mathrm{det}\times \mathrm{tr}}\\
\mls Pic ^{\mathbf{Z}}_{\det (E), \mls O'}\times \mls Pic ^{\mathbf{Z}}_{\det (E), \mls O[K]}\ar[r]^-{\simeq }& \mls Pic ^{\mathbf{Z}}_{\det (E), \mls O'}\times R\Gamma ^\Delta (BK),}
$$
$$
\xymatrix{
\mls F_{E, \mls O'\times _{\mls O}\mls O'}\ar[r]\ar[d]_-{\mathrm{det}}& \mls F_{E, \mls O'}\ar[d]_-{\mathrm{det}}\\
\mls Pic ^{\mathbf{Z}}_{\det (E), \mls O'\times _{\mls O}\mls O'}\ar[r]& \mls Pic ^{\mathbf{Z}}_{\det (E), \mls O'}.}
$$
From this and the construction of the action in \cref{P:6.5} we get statements (ii) and (iii) in \cref{T:maintheorem}.
\end{pg}

\section{An alternate proof of \cref{T:maintheorem} (i) in the case of global resolutions}\label{S:section10}

For basic facts about the  filtered derived category see  \cite[Chapter V]{Illusie}.

As in the previous section we work with a site  $\BaseSite$ and  consider a surjective map of sheaves of rings  $\mls O'\rightarrow \mls O$ with square zero kernel $K$.

\begin{pg} Let $DF(\mls O)$ denote the filtered derived category of perfect complexes $E$ equipped with locally finite decreasing filtration $F_E^\bullet $ such that $F_E^i = 0$ for $i<<0$ and $F_E^i = E$ for $i>>0$, and such that each of the graded pieces $\text{gr}^iE$ are perfect complexes (see for example \cite[Chapter V, 3.1]{Illusie}). The category $DF(\mls O)$ is a triangulated category.

There is a forgetful functor
$$
\epsilon :DF(\mls O)\rightarrow D(\mls O), \ \ (E, F_E^\bullet )\mapsto E
$$
and for each $i$ a functor
$$
\text{gr}^i:DF(\mls O)\rightarrow D(\mls O),  \ \ (E, F_E^\bullet )\mapsto \text{gr}^iE.
$$

For $(E, F_E^\bullet )\in DF(A)$ we get by the same construction as in the unfiltered case a trace map
$$
\text{tr}:E\otimes E^\vee \rightarrow \mls O,
$$
where $\mls O$ is viewed as filtered with $F_{\mls O}^i =0$ for $i<0$ and $F_{\mls O}^0 = \mls O$.   

Let $I\in D(\mls O)$ be an object viewed as a filtered object with $F_I^i =0$ for $i<0$ and $F_I^0 =I$.  
\end{pg}

\begin{prop}\label{P:traceadd} For any $u\in \mathrm{Hom}_{DF(\mls O)}(E, E\otimes I)$ we have
$$
\mathrm{tr}(\epsilon (u)) = \sum _i\mathrm{tr}(\mathrm{gr}^i(u))
$$
in $H^0(\BaseSite, I).$
\end{prop}
\begin{proof}
This is \cite[V, 3.7.7]{Illusie}.
\end{proof}

\begin{pg} 
Gabber's construction of the obstruction to deforming a perfect complex carries through in the filtered context as well.

Let $(E, F_E^\bullet )\in DF(\mls O)$ be an object with each $\text{gr}^iE$ a perfect complex.    Repeating the discussion in \cref{P:7.5}, choosing $G$ to be a filtered complex such that for all $i$ the complex $F_G^i$ is a bounded above acyclic complex of $\mls O'$-modules and
$$
F_G^i\rightarrow F_E^i
$$
is surjective.  Then $S$ is also filtered and the sequence \eqref{E:7.5.1} becomes an exact sequence of filtered complexes.  We then get a class
$$
\tilde o(E, F_E^\bullet )\in \mathrm{Ext}^2_{DF(\mls O)}(E, E\lotimes _{\mls O}K)
$$
mapping to 
$$
o(E)\in \mathrm{Ext}^2_{D(\mls O)}(E, E\lotimes K).
$$

Moreover, by the construction the class
$$
\mathrm{gr}^i(\tilde o(E, F_E^\bullet ))\in \mathrm{Ext}^2_{\mls O}(\mathrm{gr}^iE, \mathrm{gr}^iE\lotimes K)
$$
is equal to the obstruction $o(\mathrm{gr}^iE)$.
\end{pg}

\begin{pg} Now by \cref{P:traceadd} we have
$$
\mathrm{tr}(\epsilon (\tilde o(E))) = \sum \text{tr}(\text{gr}^i(\tilde o(E))).
$$
Furthermore, since each $\text{gr}^i$ is a triangulated functor we have
$$
\text{gr}^i(\tilde o(E)) = o(\text{gr}^iE).
$$
We conclude that
$$
\mathrm{tr}(o(E)) = \sum _i \mathrm{tr}(o(\text{gr}^i(E))).
$$
Now we have
$$
\mathrm{det}(E) = \otimes _i\det(\mathrm{gr}^iE)),
$$
and by \cref{R:7.1.half} we have
$$
o(\mathrm{det}(E)) = \sum _io(\det (\mathrm{gr}^iE)).
$$
We conclude:
\end{pg}

\begin{prop}\label{P:alternate} If \cref{T:maintheorem} (i)  holds for $\mathrm{gr}^iE$ for all $i$ then \cref{T:maintheorem} (i)  holds for $E$.
\end{prop}
\begin{proof}
Indeed we have
$$
\mathrm{tr}(o(E)) = \sum _i \mathrm{tr}(o(\text{gr}^i(E))) = \sum _io(\det (\mathrm{gr}^i(E))) = o(\det (E)).
$$
\end{proof}

\begin{rem} In particular, if $E$ is a strictly perfect complex then we can consider the filtration on $E$ whose successive quotients are $E^i$ and conclude that \cref{T:maintheorem} (i) holds by \cref{P:alternate} and the case of locally free sheaves (which is straightforward).
\end{rem}

\begin{appendix}

\section{$\mls D(\mathbf{S}, A_\bullet )$ and dg-modules}\label{A:appendixA}

Let $\Lambda $ be a commutative ring, let $\BaseSite$ be a site, and let $\mls O$ be a sheaf of $\Lambda $-algebras on $\BaseSite$.

For a sheaf of simplicial $\mls O$-algebras $A_\bullet $ on $\BaseSite $ there are different approaches to defining the associated $\infty $-categorical derived category of $A_\bullet $-modules.  For the convenience of the reader we summarize here how one can compare the different approaches.

Let  $\text{Sh} (\mathbf{S}, \mls O)$ denote the category of sheaves of $\mls O$-modules, and let $C(\mathbf{S}, \mls O)$ denote the category of complexes of $\mls O$-modules. 

\subsection{Differential graded modules}

For a strictly commutative differential graded algebra $\mls O$-algebra $B^\bullet $
(see for example \SPcite{061V} and \SPcite{061W}), we can consider its associated category of differential graded modules \SPcite{09JI}, which we will denote by $\mathrm{Mod}^{\dg }_{B^\bullet }$.

There is a forgetful functor
$$
\Sigma :\mathrm{Mod}_{B^\bullet }^\dg \rightarrow C(\BaseSite, \mls O),
$$
  We consider the flat model category structure on $C(\BaseSite, \mls O)$, defined in this generality in \cite[2.1.3]{LZ}.  Since the flat model category structure is monoidal, with respect to the usual tensor product of complexes, we can then use \cite[4.1]{SS} to get a model category structure on $\mathrm{Mod}_{B^\bullet }^\dg $.  

Note that $\mathrm{Mod}_{B^\bullet }^\dg $ is again a $\Lambda $-linear dg-category.  We can therefore apply the construction of \cite[1.3.1.6]{LurieHA} to get an $\infty $-category
$$
\mls D (\BaseSite, B^\bullet ):= N_\dg (\mathrm{Mod}_{B^\bullet }^{\dg, \circ  }),
$$
where $\mathrm{Mod} _{B^\bullet }^{\dg, \circ } \subset \mathrm{Mod}^\dg _{B^\bullet }$ denotes the subcategory of cofibrant-fibrant objects.

\begin{rem} One can show directly that the category $\text{Mod}^{\dg , \circ }_{B^\bullet }$ is pre-triangulated in the sense of \cite{BondalKapranov}.  From this and \cite[4.3.1]{Faonte} it follows that $\mls D_{\dg }(\mathbf{S}, B^\bullet )$ is a stable $\infty $-category.  
\end{rem}

\begin{pg}
The stable $\infty $-category $\mls D(\BaseSite, B^\bullet )$ can also be described using the dg-category of cofibrant objects $\mathrm{Mod}^{\dg, \cof }_{B^\bullet }$ as follows.

The inclusion
$$
\mathrm{Mod}_{B^\bullet }^{\dg, \circ }\hookrightarrow \mathrm{Mod}_{B^\bullet }^{\dg ,\cof}
$$
induces a morphism of $\infty $-categories
\begin{equation}\label{E:2.3.1}
N_\dg (\mathrm{Mod}_{B^\bullet }^{\dg, \circ })\rightarrow N_\dg (\mathrm{Mod}_{B^\bullet }^{\dg, \cof }).
\end{equation}
\end{pg}

\begin{lem}\label{L:A.6} The inclusion \eqref{E:2.3.1} admits a left adjoint in the sense of \cite[5.2.2.1]{LurieHT}.
\end{lem}
\begin{proof}
Let $M \in \mathrm{Mod}_{B^\bullet }^{\dg, \cof }$ be an object and let $i:M \rightarrow I $ be a trivial cofibration with $I $ fibrant.  Then $I \in \mathrm{Mod}_{B^\bullet }^{\dg, \circ }$ and for any $E\in \mathrm{Mod}_{B^\bullet }^{\dg, \circ }$ the natural map
$$
\mathrm{Hom}^\bullet _{B^\bullet }(I , E)\rightarrow \mathrm{Hom}^\bullet _{B^\bullet }(M , E)
$$
is an equivalence.  Indeed it suffices that this map induces an isomorphism on $H^i$ for each $i$, and by shifting for this it suffices to verify that it holds for $H^0$ where it holds since both sides are calculating morphisms in the homotopy category. It follows that $i$ exhibits $I $ as a localization of $M $ in the sense of \cite[5.2.7.6]{LurieHT}, and therefore by \cite[5.2.7.8]{LurieHT} there exists a left adjoint of the inclusion \eqref{E:2.3.1}.
\end{proof}

\begin{lem} Let $W$ be the collection of equivalences in $\mathrm{N}_{\dg }(\mathrm{Mod}_{B^\bullet }^{\dg, \cof })$ and consider the associated localization $\mathrm{N}_{\dg }(\mathrm{Mod}_{B^\bullet }^{\dg, \cof })[W^{-1}]$ (see \cite[1.3.4.2]{LurieHT} for the existence of this localization).  Then the composition
$$
N_\dg (\mathrm{Mod}_{B^\bullet }^{\dg, \circ })\rightarrow \mathrm{N}_{\dg }(\mathrm{Mod}_{B^\bullet }^{\dg, \cof })\rightarrow \mathrm{N}_{\dg }(\mathrm{Mod}_{B^\bullet }^{\dg, \cof })[W^{-1}]
$$
is an equivalence of $\infty $-categories
$$
N_\dg (\mathrm{Mod}_{B^\bullet }^{\dg, \circ })\rightarrow \mathrm{N}_{\dg }(\mathrm{Mod}_{B^\bullet }^{\dg, \cof} )[W^{-1}].
$$
\end{lem}
\begin{proof}
This follows from \cite[5.2.7.12]{LurieHT}.
\end{proof}

\subsection{Comparison}

 By \cite[1.3.5.3]{LurieHA} there is a model category structure on $C(\mathbf{S}, \mls O)$ in which a morphism $f:M \rightarrow N $ is a weak equivalence (resp. cofibration) if $f$ is a quasi-isomorphism (resp. term-wise injection), and fibrations are defined by the right lifting property with respect to trivial cofibrations.  We refer to this as the \emph{injective model structure}.  

The derived $\infty $-category $\mls D(\mathbf{S}, \mls O)$ is defined in \cite[1.3.5.8]{LurieHA} as the dg-nerve
$$
\mls D(\mathbf{S}, \mls O):= N_{\text{dg}}(C(\mathbf{S}, \mls O)^\circ _{\text{inj}})
$$
of the fibrant-cofibrant objects $C(\mathbf{S}, \mls O)^\circ _{\text{inj}}\subset C(\mathbf{S}, \mls O)$ with respect to the injective model structure.
By \cite[2.1.2.3]{LurieSAG} the $\infty $-category $\mls D(\mathbf{S}, \mls O)$ is identified with the hypercomplete objects in $\text{Mod}_{(\mathbf{S}, \mls O)}$.  By \cite[1.3.5.15]{LurieHA}, and the observation that every object of $C(\mathbf{S}, \mls O)$ is cofibrant, we deduce that
$$
N(C(\mathbf{S}, \mls O))[W^{-1}]\rightarrow \mls D(\mathbf{S}, \mls O)
$$
is an equivalence, where $W$ denotes the set of quasi-isomorphisms.  We therefore obtain an equivalence between $N(C(\mathbf{S}, \mls O))[W^{-1}]$ and the $\infty $-category of hypercomplete objects in $\text{Mod}_{(\mathbf{S}, \mls O)}$.  

By \cite[2.1.4 (1)]{LZ} the identity functor is  a Quillen equivalence between $C(\mathbf{S}, \mls O)$ with the flat model category structure and $C(\mathbf{S}, \mls O)$ with the injective model structure, and therefore by \cite[1.3.4.21]{LurieHA} we can also describe $\mls D(\mathbf{S}, \mls O)$ as 
$$
N(C(\mathbf{S}, \mls O)_{\text{fl}}^\cof )[W^{-1}],
$$
where $C(\mathbf{C}, \mls O)^\cof _{\text{fl}} \subset C(\mathbf{S}, \mls O)$ are the cofibrant objects with respect to the flat model category structure.   

If $A_\bullet $ is a simplicial $\mls O$-algebra with each $A_n$ flat over $\mls O$, then the corresponding differential graded algebra $N(A_\bullet )$ is a cofibrant monoid in the monoidal model category $C(\mathbf{S}, \mls O)$.  Combining this with \cite[4.3.3.17]{LurieHA} taking $B = \mls O$, using \cite[4.3.2.8]{LurieHA}, we find that 
\begin{equation}\label{E:A1}
N(\text{Mod}^{\text{dg}, \cof }_{(\mathbf{S}, N(A_\bullet ))})[W^{-1}]\simeq \text{Mod}_{A_\bullet }(\mls D(\mathbf{S}, \mls O)).
\end{equation}
That is, the $\infty $-category associated to the model category of dg-modules over $N(A_\bullet )$ is equivalent to the $\infty $-category of hypercomplete objects in $\text{Mod}_{(\mathbf{S}, A_\bullet )}$.

In fact, the assumption that the $A_n$'s are flat over $\mls O$ is unnecessary.  If $B_\bullet \rightarrow A_\bullet $ is an equivalence, then it follows from \cite[4.3.2.8]{LurieHA} (applied to the $\infty $-category of $(A_\bullet, B_\bullet )$-bimodules in the category of $B_\bullet $-modules) that the restriction functor
$$
\text{Mod}_{A_\bullet }(\mls D(\mathbf{S}, \mls O))\rightarrow \text{Mod}_{B_\bullet }(\mls D(\mathbf{S}, \mls O))
$$
is an equivalence, and similarly by \cite[4.4]{SS} the restriction
$$
\text{Mod}^{\text{dg}}_{(\mathbf{S}, N(A_\bullet ))}\rightarrow \text{Mod}^{\text{dg}}_{(\mathbf{S}, N(B_\bullet ))}
$$
is a Quillen equivalence.  Thus if $B_\bullet \rightarrow A_\bullet $ is a cofibrant replacement we see that we also have \eqref{E:A1} without assuming that $A_\bullet $ is flat.  Summarizing (using also \cite[1.3.1.17]{LurieHA}):

\begin{thm}\label{A:thm1} Let $A_\bullet $ be a simplicial $\mls O$-algebra with associated differential graded algebra $N(A_\bullet )$.  Then the $\infty $-category $N_\dg (\text{\rm Mod}_{(\mathbf{S}, N(A_\bullet ))}^{\dg, \circ })\simeq N_{\text{\rm dg}}(\text{\rm Mod}_{(S, N(A_\bullet ))}^{\text{\rm dg}, \cof })[W^{-1}]$ is equivalent to the $\infty $-category of hypercomplete objects in $\text{\rm Mod}_{(\mathbf{S}, A_\bullet )}$.
\end{thm}

\qed

\end{appendix}

\providecommand{\bysame}{\leavevmode\hbox to3em{\hrulefill}\thinspace}
\providecommand{\MR}{\relax\ifhmode\unskip\space\fi MR }
\providecommand{\MRhref}[2]{%
  \href{http://www.ams.org/mathscinet-getitem?mr=#1}{#2}
}
\providecommand{\href}[2]{#2}

\end{document}